\documentclass{amsart}%
\usepackage{amssymb}
\usepackage{amsfonts}
\usepackage{amsmath}
\usepackage{graphicx}%
\providecommand{\U}[1]{\protect \rule{.1in}{.1in}}
\newtheorem{theorem}{Theorem}
\theoremstyle{plain}

\newtheorem{corollary}{Corollary}

\newtheorem{definition}{Definition}

\newtheorem{lemma}{Lemma}

\newtheorem{remark}{Remark}

\numberwithin{equation}{section}
\begin{document}
\title[Multi-sublinear operators and {Commutators}]{Multi-sublinear operators generated by multilinear fractional
integral{\ operators} and {commutators} on the product generalized local
Morrey spaces}
\author{F. GURBUZ}
\address{ANKARA UNIVERSITY, FACULTY OF SCIENCE, DEPARTMENT OF MATHEMATICS, TANDO\u{G}AN
06100, ANKARA, TURKEY }
\email{feritgurbuz84@hotmail.com}
\urladdr{}
\thanks{}
\thanks{}
\thanks{}
\date{}
\dedicatory{ }
\begin{abstract}
The aim of this paper is to get the boundedness of certain multi-sublinear
operators generated by {multilinear }fractional integral{\ operators} on the
product generalized local Morrey spaces under generic size conditions which
are satisfied by most of the operators in harmonic analysis. We also prove
that the commutators of multilinear operators generated by local campanato
functions and multilinear fractional integral{\ operators} are also bounded on
the product generalized local Morrey spaces.

\end{abstract}
\maketitle

\section{Introduction}

The classical Morrey spaces $L_{p,\lambda}$ were introduced by Morrey
\cite{Morrey} in 1938 to study the local behavior of solutions of second order
elliptic partial differential equations (PDEs). Later, there were many
applications of Morrey space to the Navier-Stokes equations (see
\cite{Mazzucato}), the Schr\"{o}dinger equations (see \cite{Ruiz}) and the
elliptic problems with discontinuous coefficients (see \cite{Caf, Pal}).

Let ${\mathbb{R}^{n}}$ be the $n$-dimensional Euclidean space of points
$x=(x_{1},...,x_{n})$ with norm $|x|=\left(
{\textstyle \sum \nolimits_{i=1}^{n}}
x_{i}^{2}\right)  ^{1/2}$ and corresponding $m$-fold product spaces $\left(
m\in%
\mathbb{N}
\right)  $ be $\left(  {\mathbb{R}^{n}}\right)  ^{m}={\mathbb{R}^{n}%
\times \cdots \times \mathbb{R}^{n}}$. Let $B=B(x,r)$ denotes open ball centered
at $x$ of radius $r$ for $x\in{\mathbb{R}^{n}}$ and $r>0$ and $B^{c}(x,r)$ its
complement. Also $|B(x,r)|$ is the Lebesgue measure of the ball $B(x,r)$ and
$|B(x,r)|=v_{n}r^{n}$, where $v_{n}=|B(0,1)|$. For a given measurable set $E$,
we also denote the Lebesgue measure of $E$ by $\left \vert E\right \vert $. For
any given $X\subseteq{\mathbb{R}^{n}}$ and $0<p<\infty$, denote by
$L_{p}\left(  X\right)  $ the spaces of all measurable functions $f$
satisfying%
\[
\left \Vert f\right \Vert _{L_{p}\left(  X\right)  }=\left(
{\displaystyle \int \limits_{X}}
\left \vert f\left(  x\right)  \right \vert ^{p}dx\right)  ^{1/p}<\infty.
\]

We recall the definition of classical Morrey spaces $L_{p,\lambda}$ as%

\[
L_{p,\lambda}\left(  {\mathbb{R}^{n}}\right)  =\left \{  f:\left \Vert
f\right \Vert _{L_{p,\lambda}\left(  {\mathbb{R}^{n}}\right)  }=\sup
\limits_{x\in{\mathbb{R}^{n}},r>0}r^{-\frac{\lambda}{p}}\Vert f\Vert
_{L_{p}(B(x,r))}<\infty \right \}  ,
\]
where $f\in L_{p}^{loc}({\mathbb{R}^{n}})$, $0\leq \lambda \leq n$ and $1\leq
p<\infty$.

Note that $L_{p,0}=L_{p}({\mathbb{R}^{n}})$ and $L_{p,n}=L_{\infty
}({\mathbb{R}^{n}})$. If $\lambda<0$ or $\lambda>n$, then $L_{p,\lambda
}={\Theta}$, where $\Theta$ is the set of all functions equivalent to $0$ on
${\mathbb{R}^{n}}$. It is known that $L_{p,\lambda}({\mathbb{R}^{n}})$ is an
extension of $L_{p}({\mathbb{R}^{n}})$ in the sense that $L_{p,0}%
=L_{p}({\mathbb{R}^{n}})$.

We also denote by $WL_{p,\lambda}\equiv WL_{p,\lambda}({\mathbb{R}^{n}})$ the
weak Morrey space of all functions $f\in WL_{p}^{loc}({\mathbb{R}^{n}})$ for
which
\[
\left \Vert f\right \Vert _{WL_{p,\lambda}}\equiv \left \Vert f\right \Vert
_{WL_{p,\lambda}({\mathbb{R}^{n}})}=\sup_{x\in{\mathbb{R}^{n}},r>0}%
r^{-\frac{\lambda}{p}}\Vert f\Vert_{WL_{p}(B(x,r))}<\infty,
\]
where $WL_{p}(B(x,r))$ denotes the weak $L_{p}$-space of measurable functions
$f$ for which%
\[
\Vert f\Vert_{WL_{p}(B(x,r))}\equiv \Vert f\chi_{B(x,r)}\Vert_{WL_{p}%
({\mathbb{R}^{n}})}=\sup_{t>0}t\left \vert \left \{  y\in B(x,r):\left \vert
f\left(  y\right)  \right \vert >t\right \}  \right \vert ^{1/p}.
\]

For the boundedness of the Hardy--Littlewood maximal operator, the fractional
integral operator and the Calder\'{o}n--Zygmund singular integral operator on
Morrey spaces, we refer the readers to \cite{Adams, ChFra, Peetre}. For the
properties and applications of classical Morrey spaces, see \cite{ChFraL1,
ChFraL2, Gurbuz4} and the references therein.

Let $f\in L_{1}^{loc}\left(  {\mathbb{R}^{n}}\right)  $. The Hardy-Littlewood
maximal operator $M$ is defined by
\[
Mf(x)=\sup_{t>0}|B(x,t)|^{-1}\int \limits_{B(x,t)}|f(y)|dy.
\]
It's well known that $M$ is bounded on $L_{p}\left(  {\mathbb{R}^{n}}\right)
$ for $1<p\leq \infty$ and for $p=1$ weak type inequality also holds.

It is well known that the standard Calder\'{o}n-Zygmund singular integral
operator, briefly a Calder\'{o}n-Zygmund operator $\overline{T}$ has the
following integral expression
\[
\overline{T}f\left(  x\right)  =\int \limits_{{\mathbb{R}^{n}}}K\left(
x,y\right)  f\left(  y\right)  dy,
\]
for any test function $f$ and $x\notin suppf$. Here $K$ is the
Calder\'{o}n-Zygmund kernel, which is a locally integrable function defined
away from the diagonal and satisfies the following size condition:%

\[
\left \vert K\left(  x,y\right)  \right \vert \leq C\left \vert x-y\right \vert
^{-n},\text{ }x\neq y,
\]
and some continuity assumptions. Boundedness of Calder\'{o}n-Zygmund operator
$\overline{T}$ on $L_{p}\left(  {\mathbb{R}^{n}}\right)  $ for any
$1<p<\infty$ is well known.

Let $f\in L^{loc}\left(  {\mathbb{R}^{n}}\right)  $. The fractional maximal
operator $M_{\alpha}$ and the fractional integral operator $\overline
{T}_{\alpha}$ (also known as the Riesz potential $I_{\alpha}$) are defined
respectively by%
\[
M_{\alpha}f(x)=\sup_{t>0}|B(x,t)|^{-1+\frac{\alpha}{n}}\int \limits_{B(x,t)}%
|f(y)|dy,\qquad0\leq \alpha<n
\]%
\[
\overline{T}_{\alpha}f\left(  x\right)  =\int \limits_{{\mathbb{R}^{n}}}%
\frac{f\left(  y\right)  }{\left \vert x-y\right \vert ^{n-\alpha}}%
dy,\qquad0<\alpha<n.
\]
$M_{\alpha}$ and $\overline{T}_{\alpha}$ play important roles in harmonic
analysis (see \cite{LuDingY, Stein93, Torch}). Also, the fractional integral
play an essential role in PDEs. It is well known that, see \cite{Stein93} for
example, $\overline{T}_{\alpha}$ is bounded from $L_{p}\left(  {\mathbb{R}%
^{n}}\right)  $ to $L_{q}\left(  {\mathbb{R}^{n}}\right)  $ for all $p>1$ and
$\frac{1}{q}=\frac{1}{p}-\frac{\alpha}{n}>0$, and $\overline{T}_{\alpha}$ is
also of weak type $\left(  1,\frac{n}{n-\alpha}\right)  ,$ this is known as
Hardy-Littlewood Sobolev inequality. Boundedness of the fractional integral
operator $\overline{T}_{\alpha}$ on the space $M_{p,\lambda}\left(
{\mathbb{R}^{n}}\right)  $ has been studied by Spanne (published by Peetre
\cite{Peetre}) and Adams \cite{Adams}.

Recall that, for $0<\alpha<n$,%
\[
M_{\alpha}f\left(  x\right)  \leq \nu_{n}^{\frac{\alpha}{n}-1}\overline
{T}_{\alpha}\left(  \left \vert f\right \vert \right)  \left(  x\right)
\]
holds (see \cite{Li-Yang}, Remark 2.1). Hence one gets the boundedness of the
fractional maximal operator $M_{\alpha}$ from the boundedness of $\overline
{T}_{\alpha},$ where $\upsilon_{n}$ is the volume of the unit ball on
${\mathbb{R}^{n}}$.

In 1976, Coifman et al. \cite{CRW} introduced the commutator $\overline{T}%
_{b}$ generated by the Calder\'{o}n-Zygmund operator $\overline{T}$ and a
locally integrable function $b$ as follows:
\begin{equation}
\overline{T}_{b}f(x)\equiv b(x)\overline{T}f(x)-\overline{T}(bf)(x)=\int
\limits_{{\mathbb{R}^{n}}}K\left(  x,y\right)  [b(x)-b(y)]f(y)dy. \label{e3}%
\end{equation}
It is well known from \cite{CRW} that $\overline{T}_{b}$ is a bounded operator
on $L_{p}\left(  {\mathbb{R}^{n}}\right)  $, $1<p<\infty$ if and only if $b\in
BMO$ (bounded mean oscillation).

Let $b$ be a locally integrable function on ${\mathbb{R}^{n}}$. Then we shall
define the commutators for a suitable function $f$ generated by the fractional
integral operators and $b$ as follows:%
\[
\lbrack b,\overline{T}_{\alpha}]f(x)\equiv b(x)\overline{T}_{\alpha
}f(x)-\overline{T}_{\alpha}(bf)(x)=%
{\displaystyle \int \limits_{{\mathbb{R}^{n}}}}
[b(x)-b(y)]\frac{f(y)}{|x-y|^{n-\alpha}}dy,
\]
where $0<\alpha<n$.

We denote by $\overrightarrow{y}=\left(  y_{1},\ldots,y_{m}\right)  $,
$d\overrightarrow{y}=dy_{1}\ldots dy_{m}$, and by $\overrightarrow{f}$ the
$m$-tuple $\left(  f_{1},...,f_{m}\right)  $, $m$, $n$ the nonnegative
integers with $n\geq2$, $m\geq1$.

Let $\overrightarrow{f}\in L_{p_{1}}^{loc}\left(  {\mathbb{R}^{n}}\right)
\times \cdots \times L_{p_{m}}^{loc}\left(  {\mathbb{R}^{n}}\right)  $. Then
multi-sublinear fractional maximal operator $M_{\alpha}^{\left(  m\right)  }$
is defined by%
\[
M_{\alpha}^{\left(  m\right)  }\left(  \overrightarrow{f}\right)  \left(
x\right)  =\sup_{t>0}\left \vert B\left(  x,t\right)  \right \vert
^{\frac{\alpha}{n}}\left[
{\displaystyle \prod \limits_{i=1}^{m}}
\frac{1}{\left \vert B\left(  x,t\right)  \right \vert }\int \limits_{B\left(
x,t\right)  }\left \vert f_{i}\left(  y_{i}\right)  \right \vert \right]
d\overrightarrow{y},\text{ \  \ }0\leq \alpha<mn.
\]
From definition, if $\alpha=0$ then $M_{\alpha}^{\left(  m\right)  }$ is the
multi-sublinear maximal operator $M^{\left(  m\right)  }$ and also; in the
case of $m=1$, $M_{\alpha}^{\left(  m\right)  }$ is the classical fractional
maximal operator $M_{\alpha}$.

{After the work of Coifman and Meyer \cite{CM} the multilinear theory is}
received increasing attention. Multilinear {Calder\'{o}n-Zygmund operators are
studied by Grafakos-Torres \cite{Grafakos1, Grafakos2, Grafakos3*} and
Grafakos-Kalton \cite{Grafakos3} and the multilinear fractional integral
operators by Grafakos \cite{Grafakos-sm} and Kenig-Stein \cite{Kenig}.}

Firstly, recall that the $m$(multi)-linear {Calder\'{o}n-Zygmund operator
}$\overline{T}^{\left(  m\right)  }$ $\left(  m\in%
\mathbb{N}
\right)  ${\ }for test vector $\overrightarrow{f}=\left(  f_{1},\ldots
,f_{m}\right)  $ is defined by%
\[
\overline{T}^{\left(  m\right)  }\left(  \overrightarrow{f}\right)  \left(
x\right)  =%
{\displaystyle \int \limits_{\left(  {\mathbb{R}^{n}}\right)  ^{m}}}
K\left(  x,y_{1},\ldots,y_{m}\right)  \left \{
{\displaystyle \prod \limits_{i=1}^{m}}
f_{i}\left(  y_{i}\right)  \right \}  dy_{1}\cdots dy_{m},\text{ \  \  \ }x\notin%
{\displaystyle \bigcap \limits_{i=1}^{m}}
suppf_{i},
\]
where $K$ is an $m$-{Calder\'{o}n-Zygmund kernel which is a locally integrable
function defined }away from the diagonal $y_{0}=y_{1}=\cdots=y_{m}$ on
$\left(  {\mathbb{R}^{n}}\right)  ^{m+1}$ satisfying the {following size
estimate:}%
\begin{equation}
\left \vert K\left(  x,y_{1},\ldots,y_{m}\right)  \right \vert \leq \frac
{C}{\left \vert \left(  x-y_{1},\ldots,x-y_{m}\right)  \right \vert ^{mn}},
\label{10*}%
\end{equation}
for some $C>0$ and some smoothness estimates, see [20-22] for details.

The result of Grafakos and Torres \cite{Grafakos1, Grafakos3*} shows that the
multilinear {Calder\'{o}n-Zygmund operator is bounded on the product of
Lebesgue spaces.}

\begin{theorem}
\label{TeoA}\cite{Grafakos1, Grafakos3*} Let $\overline{T}^{\left(  m\right)
} $ $\left(  m\in%
\mathbb{N}
\right)  $ is a $m$-linear {Calder\'{o}n-Zygmund operator. Then, for any
numbers }$1\leq p_{1},\ldots,p_{m}<\infty$ with $\frac{1}{p}=\frac{1}{p_{1}%
}+\cdots+\frac{1}{p_{m}}$, $\overline{T}^{\left(  m\right)  }$ can be extended
to a bounded operator from $L_{p_{1}}\times \cdots \times L_{p_{m}}$ into
$L_{p}$, and bounded from $L_{1}\times \cdots \times L_{1}$ into $L_{\frac{1}%
{m},\infty}$.
\end{theorem}

In this paper we deal with another kind of multilinear operator for
$\overrightarrow{f}=\left(  f_{1},\ldots,f_{m}\right)  $, which is called
multilinear fractional integral operator as follows%
\[
\overline{T}_{\alpha}^{\left(  m\right)  }\left(  \overrightarrow{f}\right)
\left(  x\right)  =%
{\displaystyle \int \limits_{\left(  {\mathbb{R}^{n}}\right)  ^{m}}}
\frac{1}{\left \vert \left(  x-y_{1},\ldots,x-y_{m}\right)  \right \vert
^{mn-\alpha}}\left \{
{\displaystyle \prod \limits_{i=1}^{m}}
f_{i}\left(  y_{i}\right)  \right \}  d\overrightarrow{y},
\]
whose kernel is%
\begin{equation}
\left \vert K\left(  x,y_{1},\ldots,y_{m}\right)  \right \vert =\left \vert
\left(  x-y_{1},\ldots,x-y_{m}\right)  \right \vert ^{-mn+\alpha},\text{
\  \ }0<\alpha<mn, \label{1}%
\end{equation}
where $f_{1},\ldots,f_{m}:{\mathbb{R}^{n}\rightarrow \mathbb{R}}$ are
measurable and $\left \vert \left(  x-y_{1},\ldots,x-y_{m}\right)  \right \vert
=\sqrt{%
{\textstyle \sum \limits_{i=1}^{m}}
\left \vert x-y_{i}\right \vert ^{2}}$.

It is well known that multilinear fractional integral operator was first
studied by Grafakos {\cite{Grafakos-sm}. In the following result Kenig and
Stein \cite{Kenig} have proved that the multilinear }fractional integral
operator is bounded on the product of Lebesgue spaces.

\begin{theorem}
\label{Teo-1}\cite{Kenig} Let $0<\alpha<mn$, $\overline{T}_{\alpha}^{\left(
m\right)  }$ $\left(  m\in%
\mathbb{N}
\right)  $ be an $m$-linear fractional integral operator with kernel $K$
satisfying (\ref{1}) and $f_{i}\in L_{p_{i}}\left(  {\mathbb{R}^{n}}\right)
\left(  i=1,\ldots,m\right)  $ with $1\leq p_{i}\leq \infty$ and $\frac{1}%
{q}=\frac{1}{p_{1}}+\cdots+\frac{1}{p_{m}}-\frac{\alpha}{n}>0$.

$\left(  1\right)  $ If each $p_{i}>1$, then
\[
\left \Vert \overline{T}_{\alpha}^{\left(  m\right)  }\left(  \overrightarrow
{f}\right)  \right \Vert _{L_{q}\left(  {\mathbb{R}^{n}}\right)  }\leq C%
{\displaystyle \prod \limits_{i=1}^{m}}
\left \Vert f_{i}\right \Vert _{L_{p_{i}}\left(  {\mathbb{R}^{n}}\right)  };
\]

$\left(  2\right)  $ If $p_{i}=1$ for some $i$, then%
\[
\left \Vert \overline{T}_{\alpha}^{\left(  m\right)  }\left(  \overrightarrow
{f}\right)  \right \Vert _{L_{q,\infty}\left(  {\mathbb{R}^{n}}\right)  }\leq C%
{\displaystyle \prod \limits_{i=1}^{m}}
\left \Vert f_{i}\right \Vert _{L_{p_{i}}\left(  {\mathbb{R}^{n}}\right)  },
\]
here $L_{q,\infty}\left(  {\mathbb{R}^{n}}\right)  $ denotes the weak
$L_{q}\left(  {\mathbb{R}^{n}}\right)  $ space, the constant $C>0$ independent
of $\overrightarrow{f}$.
\end{theorem}

If we take $m=1$, $\overline{T}_{\alpha}^{\left(  m\right)  }$ $\left(  m\in%
\mathbb{N}
\right)  $ is the classical fractional integral operator $\overline{T}%
_{\alpha}$. Moreover, Theorem \ref{Teo-1} is the multi-version of well-known
Hardy-Littlewood-Sobolev inequality. Weighted inequalities for the multilinear
fractional integral operators have been established by Moen \cite{Moen} and
Chen and Xue \cite{Chen}.

Xu \cite{Xu} has established the boundedness of the commutators generated by
$m$-linear {Calder\'{o}n-Zygmund }singular integrals and $RBMO$ functions with
nonhomogeneity on the{\ product of Lebesgue space}. Inspired by
\cite{Grafakos1}, \cite{Grafakos3*}, \cite{Xu}, commutators $\overline
{T}_{\overrightarrow{b}}^{\left(  m\right)  }$ generated by $m$-linear
{Calder\'{o}n-Zygmund operators }$\overline{T}^{\left(  m\right)  }$ {and
local Campanato functions }$\overrightarrow{b}=\left(  b_{1},\ldots
,b_{m}\right)  ${\ is given by}%
\[
\overline{T}_{\overrightarrow{b}}^{\left(  m\right)  }\left(  \overrightarrow
{f}\right)  \left(  x\right)  =%
{\displaystyle \int \limits_{\left(  {\mathbb{R}^{n}}\right)  ^{m}}}
K\left(  x,y_{1},\ldots,y_{m}\right)  \left[
{\displaystyle \prod \limits_{i=1}^{m}}
\left[  b_{i}\left(  x\right)  -b_{i}\left(  y_{i}\right)  \right]
f_{i}\left(  y_{i}\right)  \right]  d\overrightarrow{y},
\]
where $K\left(  x,y_{1},\ldots,y_{m}\right)  $ is a $m$-linear
{Calder\'{o}n-Zygmund kernel, }$b_{i}\in LC_{q_{i},\lambda_{i}}^{\left \{
x_{0}\right \}  }({\mathbb{R}^{n}})$ (local Campanato spaces, see for
definition in Section 4) for $0\leq \lambda_{i}<\frac{1}{n}$, $i=1,\ldots,m $.
Note that $\overline{T}_{b}$ is the special case of $\overline{T}%
_{\overrightarrow{b}}^{\left(  m\right)  }$ with taking $m=1$. Similarly, let
$b_{i}\left(  i=1,\ldots,m\right)  $ be a locally integrable functions on
${\mathbb{R}^{n}} $, then the commutators generated by $m$-linear fractional
integral operators and $\overrightarrow{b}=\left(  b_{1},\ldots,b_{m}\right)
$ is given by%
\[
\overline{T}_{\alpha,\overrightarrow{b}}^{\left(  m\right)  }\left(
\overrightarrow{f}\right)  \left(  x\right)  =%
{\displaystyle \int \limits_{\left(  {\mathbb{R}^{n}}\right)  ^{m}}}
\frac{1}{\left \vert \left(  x-y_{1},\ldots,x-y_{m}\right)  \right \vert
^{mn-\alpha}}\left[
{\displaystyle \prod \limits_{i=1}^{m}}
\left[  b_{i}\left(  x\right)  -b_{i}\left(  y_{i}\right)  \right]
f_{i}\left(  y_{i}\right)  \right]  d\overrightarrow{y},
\]
where $0<\alpha<mn$, and $f_{i}$ $\left(  i=1,\ldots,m\right)  $ are suitable functions.

Suppose that $T_{\alpha}^{\left(  m\right)  }$, $\alpha \in \left(  0,mn\right)
$ represents a multilinear or a multi-sublinear operator, which satisfies that
for any $m\in%
\mathbb{N}
,$ $\overrightarrow{f}=\left(  f_{1},\ldots,f_{m}\right)  $, and $x\notin%
{\displaystyle \bigcap \limits_{i=1}^{m}}
suppf_{i},$%
\begin{equation}
\left \vert T_{\alpha}^{\left(  m\right)  }\left(  \overrightarrow{f}\right)
\left(  x\right)  \right \vert \leq c_{0}%
{\displaystyle \int \limits_{\left(  {\mathbb{R}^{n}}\right)  ^{m}}}
\frac{1}{\left \vert \left(  x-y_{1},\ldots,x-y_{m}\right)  \right \vert
^{mn-\alpha}}\left \{
{\displaystyle \prod \limits_{i=1}^{m}}
\left \vert f_{i}\left(  y_{i}\right)  \right \vert \right \}  d\overrightarrow
{y}, \label{4}%
\end{equation}
where $c_{0}$ is independent of $\overrightarrow{f},$ $x$ and each $f_{i}$
$\left(  i=1,\ldots,m\right)  $ is integrable on ${\mathbb{R}^{n}}$ with
compact support.

Condition (\ref{4}) in the case of $m=1$ was first introduced by Soria and
Weiss in \cite{SW} and is satisfied by many interesting operators in harmonic
analysis, such as the $m$-linear fractional integral operator, multi-sublinear
fractional maximal operator and so on (see \cite{LLY}, \cite{ShiTao},
\cite{SW} for details).

In this paper we will establish the boundedness of a large class of
multi-sublinear operators, including multi-sublinear fractional maximal
operators, with multilinear{\ }fractional integral{\ operators as their
special cases, and give local Campanato space estimates for commutators }on
the product generalized local Morrey spaces.

At last, throughout the paper we use the letter $C$ for a positive constant,
independent of appropriate parameters and not necessarily the same at each
occurrence. By $A\lesssim B$ we mean that $A\leq CB$ with some positive
constant $C$ independent of appropriate quantities. If $A\lesssim B$ and
$B\lesssim A$, we write $A\approx B$ and say that $A$ and $B$ are equivalent.

\section{Generalized local Morrey spaces}

After studying Morrey spaces in detail, researchers have passed to generalized
Morrey spaces. Mizuhara \cite{Miz} has given generalized Morrey spaces
$M_{p,\varphi}$ considering $\varphi=\varphi \left(  r\right)  $ instead of
$r^{\lambda}$ in the above definition of the Morrey space. Later, Guliyev
\cite{GulJIA} has defined the generalized Morrey spaces $M_{p,\varphi}$ with
normalized norm as follows:

\begin{definition}
\cite{GulJIA}\textbf{\ }Let $\varphi(x,r)$ be a positive measurable function
on ${\mathbb{R}^{n}}\times(0,\infty)$ and $1\leq p<\infty$. $M_{p,\varphi
}\equiv M_{p,\varphi}({\mathbb{R}^{n}})$ denotes the generalized Morrey space,
the space of all functions $f\in L_{p}^{loc}({\mathbb{R}^{n}})$ with finite
quasinorm
\[
\Vert f\Vert_{M_{p,\varphi}}=\sup \limits_{x\in{\mathbb{R}^{n}},r>0}%
\varphi(x,r)^{-1}|B(x,r)|^{-\frac{1}{p}}\Vert f\Vert_{L_{p}(B(x,r))}<\infty.
\]
Also $WM_{p,\varphi}\equiv WM_{p,\varphi}({\mathbb{R}^{n}})$ denotes the weak
generalized Morrey space of all functions $f\in WL_{p}^{loc}({\mathbb{R}^{n}%
})$ for which
\[
\Vert f\Vert_{WM_{p,\varphi}}=\sup \limits_{x\in{\mathbb{R}^{n}},r>0}%
\varphi(x,r)^{-1}|B(x,r)|^{-\frac{1}{p}}\Vert f\Vert_{WL_{p}(B(x,r))}<\infty.
\]

\end{definition}

Everywhere in the sequel we assume that $\inf \limits_{x\in{\mathbb{R}^{n}%
},r>0}\varphi(x,r)>0$ which makes the above spaces non-trivial, since the
spaces of bounded functions are contained in these spaces.

In \cite{GulJIA}, \cite{Gurbuz5}, \cite{Lin}, \cite{Miz}, \cite{Nakai} and
\cite{Softova}, the boundedness of the maximal operator and
Calder\'{o}n-Zygmund operator on the generalized Morrey spaces has been
obtained, respectively. For generalized Morrey spaces with nondoubling
measures see also \cite{Sawano}.

There are many papers discussing the conditions on $\varphi(x,r)$ to obtain
the boundedness of operators on the generalized Morrey spaces. For example, in
\cite{Nakai}, the following condition has been imposed on $\varphi(x,r)$:%
\begin{equation}
c^{-1}\varphi(x,r)\leq \varphi(x,t)\leq c\, \varphi(x,r), \label{a}%
\end{equation}
whenever $r\leq t\leq2r$, where $c~(\geq1)$ does not depend on $t$, $r$ and
$x\in{\mathbb{R}^{n}}$, jointly with the condition:
\begin{equation}
\int \limits_{r}^{\infty}\varphi(x,t)^{p}\frac{dt}{t}\leq C\varphi(x,r)^{p},
\label{b}%
\end{equation}
for some operators $T$ satisfying the condition (\ref{4}) (by taking $m=1$
there), where $C$ does not depend on $x\in{\mathbb{R}^{n}}$ and $r$.

In \cite{DingYZ} the boundedness of sublinear operators satisfying condition
(\ref{4}) (by taking $m=1$ there) has been proved. For the properties of
generalized Morrey spaces $M_{p,\varphi}$, see also \cite{GulJIA, Gurbuz5,
Softova,}.

Recall that the concept of the generalized local (central) Morrey space
$LM_{p,\varphi}^{\{x_{0}\}}$ has been introduced in \cite{BGGS} and studied in
\cite{Gurbuz1, Gurbuz2, Gurbuz3}.

\begin{definition}
\label{Definition2}\textbf{\ }Let $\varphi(x,r)$ be a positive measurable
function on ${\mathbb{R}^{n}}\times(0,\infty)$ and $1\leq p<\infty$. For any
fixed $x_{0}\in{\mathbb{R}^{n}}$ we denote by $LM_{p,\varphi}^{\{x_{0}%
\}}\equiv LM_{p,\varphi}^{\{x_{0}\}}({\mathbb{R}^{n}})$ the generalized local
Morrey space, the space of all functions $f\in L_{p}^{loc}({\mathbb{R}^{n}})$
with finite quasinorm
\[
\Vert f\Vert_{LM_{p,\varphi}^{\{x_{0}\}}}=\sup \limits_{r>0}\varphi
(x_{0},r)^{-1}|B(x_{0},r)|^{-\frac{1}{p}}\Vert f\Vert_{L_{p}(B(x_{0}%
,r))}<\infty.
\]
Also by $WLM_{p,\varphi}^{\{x_{0}\}}\equiv WLM_{p,\varphi}^{\{x_{0}%
\}}({\mathbb{R}^{n}})$ the weak generalized local Morrey space of all
functions $f\in WL_{p}^{loc}({\mathbb{R}^{n}})$ for which
\[
\Vert f\Vert_{WLM_{p,\varphi}^{\{x_{0}\}}}=\sup \limits_{r>0}\varphi
(x_{0},r)^{-1}|B(x_{0},r)|^{-\frac{1}{p}}\Vert f\Vert_{WL_{p}(B(x_{0}%
,r))}<\infty.
\]

\end{definition}

As in the above, everywhere in the sequel we assume that $\inf \limits_{r>0}%
\varphi(x_{0},r)>0$ for the same reasons.

According to this definition, we recover the local Morrey space $LM_{p,\lambda
}^{\{x_{0}\}}$ and the weak local Morrey space $WLM_{p,\lambda}^{\{x_{0}\}}$
under the choice $\varphi(x_{0},r)=r^{\frac{\lambda-n}{p}}$:
\[
LM_{p,\lambda}^{\{x_{0}\}}=LM_{p,\varphi}^{\{x_{0}\}}\mid_{\varphi
(x_{0},r)=r^{\frac{\lambda-n}{p}}},~~~~~~WLM_{p,\lambda}^{\{x_{0}%
\}}=WLM_{p,\varphi}^{\{x_{0}\}}\mid_{\varphi(x_{0},r)=r^{\frac{\lambda-n}{p}}%
}.
\]

The main goal of \cite{BGGS, Gurbuz1, Gurbuz2, Gurbuz3} is to give some
sufficient conditions for the boundedness of a large class of rough sublinear
operators and their commutators on the generalized local Morrey space
$LM_{p,\varphi}^{\{x_{0}\}}$. For the properties and applications of
generalized local Morrey spaces $LM_{p,\varphi}^{\{x_{0}\}}$, see \cite{BGGS,
Gurbuz1, Gurbuz2, Gurbuz3}.

Furthermore, we have the following embeddings:%
\[
M_{p,\varphi}\subset LM_{p,\varphi}^{\{x_{0}\}},\qquad \Vert f\Vert
_{LM_{p,\varphi}^{\{x_{0}\}}}\lesssim \Vert f\Vert_{M_{p,\varphi}},
\]%
\[
WM_{p,\varphi}\subset WLM_{p,\varphi}^{\{x_{0}\}},\qquad \Vert f\Vert
_{WLM_{p,\varphi}^{\{x_{0}\}}}\lesssim \Vert f\Vert_{WM_{p,\varphi}}.
\]

Now, we can give $\lambda$-central bounded mean oscillation space's historical development.

Wiener \cite{Wiener1, Wiener2} has looked for a way to describe the behavior
of a function at the infinity. The conditions which he has considered are
related to appropriate weighted $L_{q}$ spaces. Beurling \cite{Beurl} has
extended this idea and has defined a pair of dual Banach spaces $A_{q}$ and
$B_{q^{\prime}}$, where $1/q+1/q^{\prime}=1$. To be precise, $A_{q}$ is a
Banach algebra with respect to the convolution, expressed as a union of
certain weighted $L_{q}$ spaces; the space $B_{q^{\prime}}$ is expressed as
the intersection of the corresponding weighted $L_{q^{\prime}}$ spaces.
Feichtinger \cite{Feicht} has observed that the space $B_{q}$ can be described
by
\begin{equation}
\left \Vert f\right \Vert _{B_{q}}=\sup_{k\geq0}2^{-\frac{kn}{q}}\Vert f\chi
_{k}\Vert_{L_{q}({\mathbb{R}^{n}})}<\infty, \label{e21}%
\end{equation}
where $\chi_{0}$ is the characteristic function of the unit ball
$\{x\in{\mathbb{R}^{n}}:|x|\leq1\}$, $\chi_{k}$ is the characteristic function
of the annulus $\{x\in{\mathbb{R}^{n}}:2^{k-1}<|x|\leq2^{k}\}$, $k=1,2,\ldots
$. By duality, the space $A_{q}({\mathbb{R}^{n}})$, appropriately called now
the Beurling algebra , can be described by the condition%

\begin{equation}
\left \Vert f\right \Vert _{A_{q}}=\sum \limits_{k=0}^{\infty}2^{-\frac
{kn}{q^{\prime}}}\Vert f\chi_{k}\Vert_{L_{q}({\mathbb{R}^{n}})}<\infty.
\label{e22}%
\end{equation}

Let $\dot{B}_{q}({\mathbb{R}^{n}})$ and $\dot{A}_{q}({\mathbb{R}^{n}})$ be the
homogeneous versions of $B_{q}({\mathbb{R}^{n}})$ and $A_{q}({\mathbb{R}^{n}%
})$ by taking $k\in \mathbb{Z}$ in (\ref{e21}) and (\ref{e22}) instead of
$k\geq0$ there (see \cite{FuLinLu} for details).

If $\lambda<0$ or $\lambda>n$, then $LM_{p,\lambda}^{\{x_{0}\}}({\mathbb{R}%
^{n}})={\Theta}$. Note that $LM_{p,0}({\mathbb{R}^{n}})=L_{p}({\mathbb{R}^{n}%
})$ and $LM_{p,n}({\mathbb{R}^{n}})=\dot{B}_{p}({\mathbb{R}^{n}})$.
\[
\dot{B}_{p,\mu}=LM_{p,\varphi}\mid_{\varphi(0,r)=r^{\mu n}},~~~~~~W\dot
{B}_{p,\mu}=WLM_{p,\varphi}\mid_{\varphi(0,r)=r^{\mu n}}.
\]

Alvarez et al. \cite{AlvLanLakey}, in order to study the relationship between
central $BMO$ spaces and Morrey spaces, introduced $\lambda$-central bounded
mean oscillation spaces and central Morrey spaces $\dot{B}_{p,\mu}%
({\mathbb{R}^{n}})\equiv LM_{p,n+np\mu}({\mathbb{R}^{n}})$, $\mu \in
\lbrack-\frac{1}{p},0]$. If $\mu<-\frac{1}{p}$ or $\mu>0$, then $\dot
{B}_{p,\mu}({\mathbb{R}^{n}})={\Theta}$. Note that $\dot{B}_{p,-\frac{1}{p}%
}({\mathbb{R}^{n}})=L_{p}({\mathbb{R}^{n}})$ and $\dot{B}_{p,0}({\mathbb{R}%
^{n}})=\dot{B}_{p}({\mathbb{R}^{n}})$. Also define the weak central Morrey
spaces $W\dot{B}_{p,\mu}({\mathbb{R}^{n}})\equiv WLM_{p,n+np\mu}%
({\mathbb{R}^{n}})$.

The following lemma, useful in itself, shows that the quasi-norm of the local
Morrey spaces $LL_{p,\lambda}^{\{x_{0}\}}$, $\lambda \geq0$ is equivalent to
the quasi-norm $\dot{B}_{p,\lambda}^{\{x_{0}\}}({\mathbb{R}^{n}})$:%
\[
\left \Vert f\right \Vert _{\dot{B}_{p,\lambda}^{\{x_{0}\}}}=\sup_{k\in%
\mathbb{Z}
}2^{-\frac{k\lambda}{p}}\Vert f\chi_{k}\Vert_{L_{p}({\mathbb{R}^{n}})}%
<\infty,
\]
where $\chi_{k}$ is the characteristic function of the annulus $B\left(
x_{0},2^{k}\right)  \setminus B\left(  x_{0},2^{k-1}\right)  $, $k\in%
\mathbb{Z}
$.

\begin{lemma}
For $0<p\leq \infty$ and $\lambda \geq0$, the quasi-norm $\left \Vert
f\right \Vert _{LL_{p,\lambda}^{\{x_{0}\}}}$ is equivalent to the quasi-norm
$\left \Vert f\right \Vert _{\dot{B}_{p,\lambda}^{\{x_{0}\}}}$.
\end{lemma}

\begin{proof}
Let $0<p\leq \infty$, $\lambda \geq0$ and $f\in LL_{p,\lambda}^{\{x_{0}%
\}}\left(  {\mathbb{R}^{n}}\right)  $. Then, it follows that%
\[
\left \Vert f\right \Vert _{\dot{B}_{p,\lambda}^{\{x_{0}\}}}\leq \sup_{k\in%
\mathbb{Z}
}\left(  2^{k}\right)  ^{-\frac{\lambda}{p}}\Vert f\Vert_{L_{p}\left(
B\left(  x_{0},2^{k}\right)  \right)  }\leq \sup_{r>0}r^{-\frac{\lambda}{p}%
}\Vert f\Vert_{L_{p}\left(  B\left(  x_{0},r\right)  \right)  }=\left \Vert
f\right \Vert _{LL_{p,\lambda}^{\{x_{0}\}}}.
\]
On the other hand, for $0<p<\infty$, we get%
\begin{align*}
\left \Vert f\right \Vert _{LL_{p,\lambda}^{\{x_{0}\}}}^{p}  &  =\sup_{k\in%
\mathbb{Z}
}\sup_{2^{k-1}<r\leq2^{k}}r^{-\lambda}%
{\displaystyle \int \limits_{B\left(  x_{0},r\right)  }}
\left \vert f\left(  y\right)  \right \vert ^{p}dy\\
&  \leq2^{\lambda}\sup_{k\in%
\mathbb{Z}
}\left(  2^{k}\right)  ^{-\lambda}%
{\displaystyle \int \limits_{B\left(  x_{0},2^{k}\right)  }}
\left \vert f\left(  y\right)  \right \vert ^{p}dy\\
&  =2^{\lambda}\sup_{k\in%
\mathbb{Z}
}2^{-k\lambda}%
{\displaystyle \sum \limits_{d=-\infty}^{k}}
2^{d\lambda}2^{-d\lambda}%
{\displaystyle \int \limits_{B\left(  x_{0},2^{d}\right)  \setminus B\left(
x_{0},2^{d-1}\right)  }}
\left \vert f\left(  y\right)  \right \vert ^{p}dy\\
&  \leq2^{\lambda}\left(  \sup_{d\in%
\mathbb{Z}
}2^{-d\lambda}%
{\displaystyle \int \limits_{B\left(  x_{0},2^{d}\right)  \setminus B\left(
x_{0},2^{d-1}\right)  }}
\left \vert f\left(  y\right)  \right \vert ^{p}dy\right)  \left(  \sup_{k\in%
\mathbb{Z}
}2^{-k\lambda}%
{\displaystyle \sum \limits_{d=-\infty}^{k}}
2^{d\lambda}\right) \\
&  =\frac{2^{\lambda}}{1-2^{-\lambda}}\left \Vert f\right \Vert _{\dot
{B}_{p,\lambda}^{\{x_{0}\}}}^{p}.
\end{align*}
Thus, for $0<p<\infty$, we have%
\[
\left \Vert f\right \Vert _{LL_{p,\lambda}^{\{x_{0}\}}}\leq2^{\frac{\lambda}{p}%
}\left(  1-2^{-\lambda}\right)  ^{-\frac{1}{p}}\left \Vert f\right \Vert
_{\dot{B}_{p,\lambda}^{\{x_{0}\}}}.
\]
Also, for $p=\infty$ we have%
\[
\left \Vert f\right \Vert _{LL_{\infty,\lambda}^{\{x_{0}\}}}\leq \left \Vert
f\right \Vert _{\dot{B}_{\infty,\lambda}^{\{x_{0}\}}}.
\]

\end{proof}

In the case of $\lambda=n$, the quasi-norms $\left \Vert f\right \Vert _{\dot
{B}_{p,\lambda}^{\{x_{0}\}}}$ have been considered by Beurling \cite{Beurl}
and Feichtinger \cite{Feicht}.

Closely related to the above results, in this paper we prove the boundedness
of the multi-sublinear operators $T_{\alpha}^{\left(  m\right)  }$ $\left(
m\in%
\mathbb{N}
\right)  $, $\alpha \in \left(  0,mn\right)  $ satisfying condition (\ref{4})
from product generalized local Morrey space $LM_{p_{1},\varphi_{1}}%
^{\{x_{0}\}}\times \cdots \times LM_{p_{m},\varphi_{m}}^{\{x_{0}\}}$ to
$LM_{q,\varphi}^{\{x_{0}\}}$, if $1<p_{1},\ldots,p_{m}<\infty$, $\frac{1}{p}=%
{\displaystyle \sum \limits_{i=1}^{m}}
\frac{1}{p_{i}}$, $\frac{1}{q_{i}}=\frac{1}{p_{i}}-\frac{\alpha}{mn}$ and
$\frac{1}{q}=%
{\displaystyle \sum \limits_{i=1}^{m}}
\frac{1}{q_{i}}=\frac{1}{p}-\frac{\alpha}{n}$, and from the space
$LM_{p_{1},\varphi_{1}}^{\{x_{0}\}}\times \cdots \times LM_{p_{m},\varphi_{m}%
}^{\{x_{0}\}}$ to the weak space $WLM_{q,\varphi}^{\{x_{0}\}}$, if $1\leq
p_{1},\ldots,p_{m}<\infty$, $\frac{1}{p}=%
{\displaystyle \sum \limits_{i=1}^{m}}
\frac{1}{p_{i}}$, $\frac{1}{q_{i}}=\frac{1}{p_{i}}-\frac{\alpha}{mn}$,
$\frac{1}{q}=%
{\displaystyle \sum \limits_{i=1}^{m}}
\frac{1}{q_{i}}=\frac{1}{p}-\frac{\alpha}{n}$ and at least one exponent
$p_{i}$ $\left(  i=1,\ldots,m\right)  $ equals $1$. In the case of $b_{i}\in
LC_{q_{i},\lambda_{i}}^{\left \{  x_{0}\right \}  }({\mathbb{R}^{n}})$ for
$0\leq \lambda_{i}<\frac{1}{n}$, $i=1,\ldots,m$, we find the sufficient
conditions on $(\varphi_{1},\ldots,\varphi_{m},\varphi)$ which ensures the
boundedness of the commutator operators $T_{\alpha,\overrightarrow{b}%
}^{\left(  m\right)  }$ from $LM_{p_{1},\varphi_{1}}^{\{x_{0}\}}\times
\cdots \times LM_{p_{m},\varphi_{m}}^{\{x_{0}\}}$ to $LM_{q,\varphi}%
^{\{x_{0}\}}$, where $\frac{1}{q}=\sum \limits_{i=1}^{m}\frac{1}{p_{i}}%
+\sum \limits_{i=1}^{m}\frac{1}{q_{i}}-\frac{\alpha}{n}$. In fact, in this
paper the results of \cite{Gurbuz1, Gurbuz2, Gurbuz3} (by taking $\Omega
\equiv1$ there) will be generalized to the multilinear case; we omit the
details here. But, the techniques and non-trivial estimates which have been
used in the proofs of our main results are quite different from \cite{Gurbuz1,
Gurbuz2}. For example, using inequality about the weighted Hardy operator
$H_{w}$ in \cite{Gurbuz1, Gurbuz2}, in this paper we will only use the
following relationship between essential supremum and essential infimum%
\begin{equation}
\left(  \operatorname*{essinf}\limits_{x\in E}f\left(  x\right)  \right)
^{-1}=\operatorname*{esssup}\limits_{x\in E}\frac{1}{f\left(  x\right)  },
\label{5}%
\end{equation}
where $f$ is any real-valued nonnegative function and measurable on $E$ (see
\cite{Wheeden-Zygmund}, page 143).

\begin{remark}
Our results in this paper remain true for the inhomogeneous versions of local
Campanato spaces $LC_{q,\lambda}^{\left \{  x_{0}\right \}  }\left(
{\mathbb{R}^{n}}\right)  $ for $0\leq \lambda<\frac{1}{n}$ and generalized
local Morrey spaces $LM_{p,\varphi}^{\{x_{0}\}}$.
\end{remark}

\section{Boundedness of the multi-sublinear operators $T_{\alpha}^{\left(
m\right)  }$ $\left(  m\in%
\mathbb{N}
\right)  $ on the product spaces $LM_{p_{1},\varphi_{1}}^{\{x_{0}\}}%
\times \cdots \times LM_{p_{m},\varphi_{m}}^{\{x_{0}\}}$}

In this section we prove the boundedness of the operator $T_{\alpha}^{\left(
m\right)  }$ $\left(  m\in%
\mathbb{N}
\right)  $, $\alpha \in \left(  0,mn\right)  $ satisfying condition (\ref{4}) on
the product generalized local Morrey spaces $LM_{p_{1},\varphi_{1}}%
^{\{x_{0}\}}\times \cdots \times LM_{p_{m},\varphi_{m}}^{\{x_{0}\}}$ by using
(\ref{5}) and the following Theorem \ref{Teo2}.

We first prove the following Theorem \ref{Teo2}.

\begin{theorem}
\label{Teo2}Let $x_{0}\in{\mathbb{R}^{n}}$, $0<\alpha<mn$ and $1\leq
p_{i}<\frac{mn}{\alpha}$ with $\frac{1}{p}=%
{\displaystyle \sum \limits_{i=1}^{m}}
\frac{1}{p_{i}}$, $\frac{1}{q_{i}}=\frac{1}{p_{i}}-\frac{\alpha}{mn}$ and
$\frac{1}{q}=%
{\displaystyle \sum \limits_{i=1}^{m}}
\frac{1}{q_{i}}=\frac{1}{p}-\frac{\alpha}{n}$. Let $T_{\alpha}^{\left(
m\right)  }$ $\left(  m\in%
\mathbb{N}
\right)  $ be a multi-sublinear operator satisfying condition (\ref{4}),
bounded from $L_{p_{1}}\times \cdots \times L_{p_{m}}$ into $L_{q}$ for
$p_{i}>1$ $\left(  i=1,\ldots,m\right)  $, and bounded from $L_{p_{1}}%
\times \cdots \times L_{p_{m}}$ into the weak space $WL_{q}$, and at least one
exponent $p_{i}$ $\left(  i=1,\ldots,m\right)  $ equals $1$.

Then, for $p_{i}>1$ $\left(  i=1,\ldots,m\right)  $ the inequality
\begin{equation}
\left \Vert T_{\alpha}^{\left(  m\right)  }\left(  \overrightarrow{f}\right)
\right \Vert _{L_{q}\left(  B_{r}\right)  }\lesssim r^{\frac{n}{q}}%
\int \limits_{2r}^{\infty}%
{\displaystyle \prod \limits_{i=1}^{m}}
\left \Vert f_{i}\right \Vert _{L_{p_{i}}\left(  B\left(  x_{0},t\right)
\right)  }t^{-\frac{n}{q}-1}dt \label{f}%
\end{equation}
holds for any ball $B_{r}=B\left(  x_{0},r\right)  $ and for all
$\overrightarrow{f}\in L_{p_{1}}^{loc}\left(  {\mathbb{R}^{n}}\right)
\times \cdots \times L_{p_{m}}^{loc}\left(  {\mathbb{R}^{n}}\right)  $.

Moreover, if at least one exponent $p_{i}$ $\left(  i=1,\ldots,m\right)  $
equals $1$, the inequality%
\begin{equation}
\left \Vert T_{\alpha}^{\left(  m\right)  }\left(  \overrightarrow{f}\right)
\right \Vert _{WL_{q}\left(  B_{r}\right)  }\lesssim r^{\frac{n}{q}}%
\int \limits_{2r}^{\infty}%
{\displaystyle \prod \limits_{i=1}^{m}}
\left \Vert f_{i}\right \Vert _{L_{p_{i}}\left(  B\left(  x_{0},t\right)
\right)  }t^{-\frac{n}{q}-1}dt \label{f*}%
\end{equation}
holds for any ball $B_{r}=B\left(  x_{0},r\right)  $ and for all
$\overrightarrow{f}\in L_{p_{1}}^{loc}\left(  {\mathbb{R}^{n}}\right)
\times \cdots \times L_{p_{m}}^{loc}\left(  {\mathbb{R}^{n}}\right)  $.
\end{theorem}

\begin{proof}
In order to simplify the proof, we consider only the situation when $m=2$.
Actually, a similar procedure works for all $m\in%
\mathbb{N}
$. Thus, without loss of generality, it is sufficient to show that the
conclusion holds for $T_{\alpha}^{\left(  2\right)  }\left(  \overrightarrow
{f}\right)  =T_{\alpha}^{\left(  2\right)  }\left(  f_{1},f_{2}\right)  $.

We just consider the case $p_{i}>1$ for $i=1,2$. For any $x_{0}\in
{\mathbb{R}^{n}}$, set $B_{r}=B\left(  x_{0},r\right)  $ for the ball centered
at $x_{0}$ and of radius $r$ and $B_{2r}=B\left(  x_{0},2r\right)  $. Indeed,
we also decompose $f_{i}$ as $f_{i}\left(  y_{i}\right)  =$ $f_{i}\left(
y_{i}\right)  \chi_{B_{2r}}+f_{i}\left(  y_{i}\right)  \chi_{\left(
B_{2r}\right)  ^{c}}$ for $i=1,2$. And, we write $f_{1}=f_{1}^{0}%
+f_{1}^{\infty}$ and $f_{2}=f_{2}^{0}+f_{2}^{\infty}$, where \ $f_{i}%
^{0}=f_{i}\chi_{B_{2r}}$, \ $f_{i}^{\infty}=f_{i}\chi_{\left(  B_{2r}\right)
^{c}}$, for $i=1,2$. Thus, we have%
\begin{align*}
\left \Vert T_{\alpha}^{\left(  2\right)  }\left(  f_{1},f_{2}\right)
\right \Vert _{L_{q}\left(  B_{r}\right)  }  &  \leq \left \Vert T_{\alpha
}^{\left(  2\right)  }\left(  f_{1}^{0},f_{2}^{0}\right)  \right \Vert
_{L_{q}\left(  B_{r}\right)  }+\left \Vert T_{\alpha}^{\left(  2\right)
}\left(  f_{1}^{0},f_{2}^{\infty}\right)  \right \Vert _{L_{q}\left(
B_{r}\right)  }\\
&  +\left \Vert T_{\alpha}^{\left(  2\right)  }\left(  f_{1}^{\infty},f_{2}%
^{0}\right)  \right \Vert _{L_{q}\left(  B_{r}\right)  }+\left \Vert T_{\alpha
}^{\left(  2\right)  }\left(  f_{1}^{\infty},f_{2}^{\infty}\right)
\right \Vert _{L_{q}\left(  B_{r}\right)  }\\
&  =I_{1}+I_{2}+I_{3}+I_{4}.
\end{align*}

Firstly, we use the boundedness of $T_{\alpha}^{\left(  2\right)  }$ from
$L_{p_{1}}\times L_{p_{2}}$ into $L_{q}$ to estimate $I_{1}$, and we obtain%
\begin{align*}
I_{1}  &  =\left \Vert T_{\alpha}^{\left(  2\right)  }\left(  f_{1}^{0}%
,f_{2}^{0}\right)  \right \Vert _{L_{q}\left(  B_{r}\right)  }\lesssim
\left \Vert f_{1}\right \Vert _{L_{p_{1}}\left(  B_{2r}\right)  }\left \Vert
f_{2}\right \Vert _{L_{p_{2}}\left(  B_{2r}\right)  }\\
&  \lesssim r^{\frac{n}{q}}\left \Vert f_{1}\right \Vert _{L_{p_{1}}\left(
B_{2r}\right)  }\left \Vert f_{2}\right \Vert _{L_{p_{2}}\left(  B_{2r}\right)
}\int \limits_{2r}^{\infty}\frac{dt}{t^{\frac{n}{q}+1}}\\
&  \leq r^{\frac{n}{q}}\int \limits_{2r}^{\infty}\prod \limits_{i=1}%
^{2}\left \Vert f_{i}\right \Vert _{L_{p_{i}}\left(  B_{t}\right)  }\frac
{dt}{t^{\frac{n}{q}+1}}.
\end{align*}

Secondly, it is clear that $\left \vert \left(  x_{0}-y_{1},\text{ }x_{0}%
-y_{2}\right)  \right \vert ^{2n-\alpha}\geq \left \vert x_{0}-y_{2}\right \vert
^{2n-\alpha}$. By the condition (\ref{4}) with $m=2$, H\"{o}lder's inequality,
the estimate of $I_{2}$ can be obtained as follows:%
\begin{align*}
\left \vert T_{\alpha}^{\left(  2\right)  }\left(  f_{1}^{0},f_{2}^{\infty
}\right)  \left(  x\right)  \right \vert  &  \lesssim%
{\displaystyle \int \limits_{{\mathbb{R}^{n}}}}
{\displaystyle \int \limits_{{\mathbb{R}^{n}}}}
\frac{\left \vert f_{1}^{0}\left(  y_{1}\right)  \right \vert \left \vert
f_{2}^{\infty}\left(  y_{2}\right)  \right \vert }{\left \vert \left(
x-y_{1},x-y_{2}\right)  \right \vert ^{2n-\alpha}}dy_{1}dy_{2}\\
&  \lesssim \int \limits_{B_{2r}}\left \vert f_{1}\left(  y_{1}\right)
\right \vert dy_{1}\int \limits_{\left(  B_{2r}\right)  ^{c}}\frac{\left \vert
f_{2}\left(  y_{2}\right)  \right \vert }{\left \vert x_{0}-y_{2}\right \vert
^{2n-\alpha}}dy_{2}\\
&  \approx \int \limits_{B_{2r}}\left \vert f_{1}\left(  y_{1}\right)
\right \vert dy_{1}\int \limits_{\left(  B_{2r}\right)  ^{c}}\left \vert
f_{2}\left(  y_{2}\right)  \right \vert \int \limits_{\left \vert x_{0}%
-y_{2}\right \vert }^{\infty}\frac{dt}{t^{2n-\alpha+1}}dy_{2}\\
&  \lesssim \left \Vert f_{1}\right \Vert _{L_{p_{1}}\left(  B_{2r}\right)
}\left \vert B_{2r}\right \vert ^{1-\frac{1}{p_{1}}}%
{\displaystyle \int \limits_{2r}^{\infty}}
\left \Vert f_{2}\right \Vert _{L_{p_{2}}\left(  B_{t}\right)  }\left \vert
B_{t}\right \vert ^{1-\frac{1}{p_{2}}}\frac{dt}{t^{2n-\alpha+1}}\\
&  \lesssim \int \limits_{2r}^{\infty}\prod \limits_{i=1}^{2}\left \Vert
f_{i}\right \Vert _{L_{p_{i}}\left(  B_{t}\right)  }\frac{dt}{t^{\frac{n}{q}%
+1}},
\end{align*}
where $\frac{1}{p}=\frac{1}{p_{1}}+\frac{1}{p_{2}}$. Thus, the inequality
\begin{equation}
I_{2}=\left \Vert T_{\alpha}^{\left(  2\right)  }\left(  f_{1}^{0}%
,f_{2}^{\infty}\right)  \right \Vert _{L_{q}\left(  B_{r}\right)  }\lesssim
r^{\frac{n}{q}}\int \limits_{2r}^{\infty}\prod \limits_{i=1}^{2}\left \Vert
f_{i}\right \Vert _{L_{p_{i}}\left(  B_{t}\right)  }\frac{dt}{t^{\frac{n}{q}%
+1}} \label{10}%
\end{equation}
is valid.

Similarly, $I_{3}$ has the same estimate above, here we omit the details, thus
the inequality%
\[
I_{3}=\left \Vert T_{\alpha}^{\left(  2\right)  }\left(  f_{1}^{\infty}%
,f_{2}^{0}\right)  \right \Vert _{L_{q}\left(  B_{r}\right)  }\lesssim
r^{\frac{n}{q}}\int \limits_{2r}^{\infty}\prod \limits_{i=1}^{2}\left \Vert
f_{i}\right \Vert _{L_{p_{i}}\left(  B_{t}\right)  }\frac{dt}{t^{\frac{n}{q}%
+1}}.
\]
is valid.

At last, we consider the term $I_{4}$. Note that $\left \vert \left(
x_{0}-y_{1},\text{ }x_{0}-y_{2}\right)  \right \vert ^{2n-\alpha}\geq \left \vert
x_{0}-y_{1}\right \vert ^{n-\frac{\alpha}{2}}\left \vert x_{0}-y_{2}\right \vert
^{n-\frac{\alpha}{2}}$. Using the condition (\ref{4}) with $m=2$ and by
H\"{o}lder's inequality, we get%

\begin{align*}
&  \left \vert T_{\alpha}^{\left(  2\right)  }\left(  f_{1}^{\infty}%
,f_{2}^{\infty}\right)  \left(  x\right)  \right \vert \\
&  \lesssim%
{\displaystyle \int \limits_{\mathbb{R} ^{n}}}
{\displaystyle \int \limits_{\mathbb{R} ^{n}}}
\frac{\left \vert f_{1}\left(  y_{1}\right)  \chi_{\left(  B_{2r}\right)  ^{c}%
}\right \vert \left \vert f_{2}\left(  y_{2}\right)  \chi_{\left(
B_{2r}\right)  ^{c}}\right \vert }{\left \vert \left(  x_{0}-y_{1},x_{0}%
-y_{2}\right)  \right \vert ^{2n-\alpha}}dy_{1}dy_{2}\\
&  \lesssim%
{\displaystyle \int \limits_{\left(  B_{2r}\right)  ^{c}}}
{\displaystyle \int \limits_{\left(  B_{2r}\right)  ^{c}}}
\frac{\left \vert f_{1}\left(  y_{1}\right)  \right \vert \left \vert
f_{2}\left(  y_{2}\right)  \right \vert }{\left \vert x_{0}-y_{1}\right \vert
^{n-\frac{\alpha}{2}}\left \vert x_{0}-y_{2}\right \vert ^{n-\frac{\alpha}{2}}%
}dy_{1}dy_{2}\\
&  \lesssim%
{\displaystyle \sum \limits_{j=1}^{\infty}}
{\displaystyle \prod \limits_{i=1}^{2}}
{\displaystyle \int \limits_{B_{2^{j+1}r}\backslash B_{2^{j}r}}}
\frac{\left \vert f_{i}\left(  y_{i}\right)  \right \vert }{\left \vert
x_{0}-y_{i}\right \vert ^{n-\frac{\alpha}{2}}}dy_{i}\\
&  \lesssim%
{\displaystyle \sum \limits_{j=1}^{\infty}}
{\displaystyle \prod \limits_{i=1}^{2}}
\left(  2^{j}r\right)  ^{-n+\frac{\alpha}{2}}%
{\displaystyle \int \limits_{B_{2^{j+1}r}}}
\left \vert f_{i}\left(  y_{i}\right)  \right \vert dy_{i}\\
&  \lesssim%
{\displaystyle \sum \limits_{j=1}^{\infty}}
\left(  2^{j}r\right)  ^{-2n+\alpha}%
{\displaystyle \prod \limits_{i=1}^{2}}
\left \Vert f_{i}\right \Vert _{L_{p_{i}}(B_{2^{j+1}r})}\left \vert B_{2^{j+1}%
r}\right \vert ^{1-\frac{1}{p_{i}}}\\
&  \lesssim%
{\displaystyle \sum \limits_{j=1}^{\infty}}
{\displaystyle \int \limits_{2^{j+1}r}^{2^{j+2}r}}
\left(  2^{j+1}r\right)  ^{-2n+\alpha-1}%
{\displaystyle \prod \limits_{i=1}^{2}}
\left \Vert f_{i}\right \Vert _{L_{p_{i}}(B_{2^{j+1}r})}\left \vert B_{2^{j+1}%
r}\right \vert ^{1-\frac{1}{p_{i}}}dt\\
&  \lesssim%
{\displaystyle \sum \limits_{j=1}^{\infty}}
{\displaystyle \int \limits_{2^{j+1}r}^{2^{j+2}r}}
{\displaystyle \prod \limits_{i=1}^{2}}
\left \Vert f_{i}\right \Vert _{L_{p_{i}}(B_{t})}\left \vert B_{t}\right \vert
^{1-\frac{1}{p_{i}}}\frac{dt}{t^{2n+1-\alpha}}\\
&  \lesssim%
{\displaystyle \int \limits_{2r}^{\infty}}
\prod \limits_{i=1}^{2}\left \Vert f_{i}\right \Vert _{L_{p_{i}}\left(
B_{t}\right)  }\left \vert B_{t}\right \vert ^{2-\left(  \frac{1}{p_{1}}%
+\frac{1}{p_{2}}\right)  }\frac{dt}{t^{2n+1-\alpha}}\\
&  \lesssim%
{\displaystyle \int \limits_{2r}^{\infty}}
\prod \limits_{i=1}^{2}\left \Vert f_{i}\right \Vert _{L_{p_{i}}\left(
B_{t}\right)  }\frac{dt}{t^{\frac{n}{q}+1}}.
\end{align*}
Moreover, for $p_{1}$, $p_{2}\in \left[  1,\infty \right)  $ the inequality
\begin{equation}
I_{4}=\left \Vert T_{\alpha}^{\left(  2\right)  }\left(  f_{1}^{\infty}%
,f_{2}^{\infty}\right)  \right \Vert _{L_{q}\left(  B_{r}\right)  }\lesssim
r^{\frac{n}{q}}\int \limits_{2r}^{\infty}\prod \limits_{i=1}^{2}\left \Vert
f_{i}\right \Vert _{L_{p_{i}}\left(  B_{t}\right)  }\frac{dt}{t^{\frac{n}{q}%
+1}} \label{0}%
\end{equation}
is valid.

By combining the above inequalities for $I_{1},I_{2},I_{3}$ and $I_{4}$ we
obtain%
\[
\left \Vert T_{\alpha}^{\left(  2\right)  }\left(  f_{1},f_{2}\right)
\right \Vert _{L_{q}\left(  B_{r}\right)  }\lesssim r^{\frac{n}{q}}%
\int \limits_{2r}^{\infty}\prod \limits_{i=1}^{2}\left \Vert f_{i}\right \Vert
_{L_{p_{i}}\left(  B_{t}\right)  }\frac{dt}{t^{\frac{n}{q}+1}}.
\]

For the proof of the inequality (\ref{f*}), by a similar argument as in the
proof of (\ref{f}) and paying attention to the fact that $\overrightarrow
{f}\rightarrow T_{\alpha}^{\left(  m\right)  }\left(  \overrightarrow
{f}\right)  $ is bounded from $L_{p_{1}}\times \cdots \times L_{p_{m}}$ to
$WL_{q}$, we can similarly prove (\ref{f*}) so we omit the details here, which
completes the proof.
\end{proof}

In the following theorem, which is one of our main results, we get the
boundedness of the multi-sublinear operator $T_{\alpha}^{\left(  m\right)  }$
$\left(  m\in%
\mathbb{N}
\right)  $, $\alpha \in \left(  0,mn\right)  $ satisfying condition (\ref{4}) on
the product generalized local Morrey spaces $LM_{p_{1},\varphi_{1}}%
^{\{x_{0}\}}\times \cdots \times LM_{p_{m},\varphi_{m}}^{\{x_{0}\}}$.

\begin{theorem}
\label{teo9}(Our main result) Let $x_{0}\in{\mathbb{R}^{n}}$, $0<\alpha<mn$
and $1\leq p_{i}<\frac{mn}{\alpha}$ with $\frac{1}{p}=%
{\displaystyle \sum \limits_{i=1}^{m}}
\frac{1}{p_{i}}$, $\frac{1}{q_{i}}=\frac{1}{p_{i}}-\frac{\alpha}{mn}$ and
$\frac{1}{q}=%
{\displaystyle \sum \limits_{i=1}^{m}}
\frac{1}{q_{i}}=\frac{1}{p}-\frac{\alpha}{n}$. Let $T_{\alpha}^{\left(
m\right)  }$ $\left(  m\in%
\mathbb{N}
\right)  $ be a multi-sublinear operator satisfying condition (\ref{4}),
bounded from $L_{p_{1}}\times \cdots \times L_{p_{m}}$ into $L_{q}$ for
$p_{i}>1$ $\left(  i=1,\ldots,m\right)  $, and bounded from $L_{p_{1}}%
\times \cdots \times L_{p_{m}}$ into the weak space $WL_{q}$, and at least one
exponent $p_{i}$ $\left(  i=1,\ldots,m\right)  $ equals $1$. If functions
$\varphi,\varphi_{i}:{\mathbb{R}^{n}\times}\left(  0,\infty \right)
\rightarrow \left(  0,\infty \right)  ,$ $\left(  i=1,\ldots,m\right)  $ and
$(\varphi_{1},\ldots,\varphi_{m},\varphi)$ satisfy the condition%
\begin{equation}
\int \limits_{r}^{\infty}\frac{\operatorname*{essinf}\limits_{t<\tau<\infty}%
{\displaystyle \prod \limits_{i=1}^{m}}
\varphi_{i}(x_{0},\tau)\tau^{\frac{n}{p}}}{t^{\frac{n}{q}+1}}dt\leq
C\varphi(x_{0},r), \label{316}%
\end{equation}
where $C$ does not depend on $r$.

Then the operator $T_{\alpha}^{\left(  m\right)  }$ is bounded from product
space $LM_{p_{1},\varphi_{1}}^{\{x_{0}\}}\times \cdots \times LM_{p_{m}%
,\varphi_{m}}^{\{x_{0}\}}$ to $LM_{q,\varphi}^{\{x_{0}\}}$ for $p_{i}>1$
$\left(  i=1,\ldots,m\right)  $ and from product $LM_{p_{1},\varphi_{1}%
}^{\{x_{0}\}}\times \cdots \times LM_{p_{m},\varphi_{m}}^{\{x_{0}\}}$ to
$WLM_{q,\varphi}^{\{x_{0}\}}$ for $p_{i}\geq1$ $\left(  i=1,\ldots,m\right)
$. Moreover, we have for $p_{i}>1$ $\left(  i=1,\ldots,m\right)  $
\begin{equation}
\left \Vert T_{\alpha}^{\left(  m\right)  }\left(  \overrightarrow{f}\right)
\right \Vert _{LM_{q,\varphi}^{\{x_{0}\}}}\lesssim%
{\displaystyle \prod \limits_{i=1}^{m}}
\left \Vert f_{i}\right \Vert _{LM_{p_{i},\varphi_{i}}^{\{x_{0}\}}}, \label{3-1}%
\end{equation}
and for $p_{i}\geq1\left(  i=1,\ldots,m\right)  $
\begin{equation}
\left \Vert T_{\alpha}^{\left(  m\right)  }\left(  \overrightarrow{f}\right)
\right \Vert _{WLM_{q,\varphi}^{\{x_{0}\}}}\lesssim%
{\displaystyle \prod \limits_{i=1}^{m}}
\left \Vert f_{i}\right \Vert _{LM_{p_{i},\varphi_{i}}^{\{x_{0}\}}}. \label{3-2}%
\end{equation}

\end{theorem}

\begin{proof}
Since $\overrightarrow{f}\in LM_{p_{1},\varphi_{1}}^{\{x_{0}\}}\times
\cdots \times LM_{p_{m},\varphi_{m}}^{\{x_{0}\}}$, by (\ref{5}) and the
non-decreasing, with respect to $t$, of the norm $%
{\displaystyle \prod \limits_{i=1}^{m}}
\left \Vert f_{i}\right \Vert _{L_{p_{i}}\left(  B_{t}\right)  }$, we get%
\begin{align}
&  \frac{%
{\displaystyle \prod \limits_{i=1}^{m}}
\left \Vert f_{i}\right \Vert _{L_{p_{i}}\left(  B_{t}\right)  }}%
{\operatorname*{essinf}\limits_{0<t<\tau<\infty}%
{\displaystyle \prod \limits_{i=1}^{m}}
\varphi_{i}(x_{0},\tau)\tau^{\frac{n}{p}}}\nonumber \\
&  \leq \operatorname*{esssup}\limits_{0<t<\tau<\infty}\frac{%
{\displaystyle \prod \limits_{i=1}^{m}}
\left \Vert f_{i}\right \Vert _{L_{p_{i}}\left(  B_{t}\right)  }}{%
{\displaystyle \prod \limits_{i=1}^{m}}
\varphi_{i}(x_{0},\tau)\tau^{\frac{n}{p}}}\nonumber \\
&  \leq \operatorname*{esssup}\limits_{0<\tau<\infty}\frac{%
{\displaystyle \prod \limits_{i=1}^{m}}
\left \Vert f_{i}\right \Vert _{L_{p_{i}}\left(  B_{\tau}\right)  }}{%
{\displaystyle \prod \limits_{i=1}^{m}}
\varphi_{i}(x_{0},\tau)\tau^{\frac{n}{p}}}\nonumber \\
&  \leq%
{\displaystyle \prod \limits_{i=1}^{m}}
\left \Vert f_{i}\right \Vert _{LM_{p_{i},\varphi_{i}}^{\{x_{0}\}}}. \label{9}%
\end{align}
For $1<p_{1},\ldots,p_{m}<\infty$, since $(\varphi_{1},\ldots,\varphi
_{m},\varphi)$ satisfies (\ref{316}), we have%
\begin{align}
&  \int \limits_{r}^{\infty}%
{\displaystyle \prod \limits_{i=1}^{m}}
\left \Vert f_{i}\right \Vert _{L_{p_{i}}\left(  B_{t}\right)  }t^{-\frac{n}%
{q}-1}dt\nonumber \\
&  \leq \int \limits_{r}^{\infty}\frac{%
{\displaystyle \prod \limits_{i=1}^{m}}
\left \Vert f_{i}\right \Vert _{L_{p_{i}}\left(  B_{t}\right)  }}%
{\operatorname*{essinf}\limits_{t<\tau<\infty}%
{\displaystyle \prod \limits_{i=1}^{m}}
\varphi_{i}(x_{0},\tau)\tau^{\frac{n}{p}}}\frac{\operatorname*{essinf}%
\limits_{t<\tau<\infty}%
{\displaystyle \prod \limits_{i=1}^{m}}
\varphi_{i}(x_{0},\tau)\tau^{\frac{n}{p}}}{t^{\frac{n}{q}}}\frac{dt}%
{t}\nonumber \\
&  \leq C%
{\displaystyle \prod \limits_{i=1}^{m}}
\left \Vert f_{i}\right \Vert _{LM_{p_{i},\varphi_{i}}^{\{x_{0}\}}}%
\int \limits_{r}^{\infty}\frac{\operatorname*{essinf}\limits_{t<\tau<\infty}%
{\displaystyle \prod \limits_{i=1}^{m}}
\varphi_{i}(x_{0},\tau)\tau^{\frac{n}{p}}}{t^{\frac{n}{q}}}\frac{dt}%
{t}\nonumber \\
&  \leq C%
{\displaystyle \prod \limits_{i=1}^{m}}
\left \Vert f_{i}\right \Vert _{LM_{p_{i},\varphi_{i}}^{\{x_{0}\}}}\varphi
(x_{0},r). \label{14}%
\end{align}
Then by (\ref{f}) and (\ref{14}), we get%
\begin{align*}
\left \Vert T_{\alpha}^{\left(  m\right)  }\left(  \overrightarrow{f}\right)
\right \Vert _{LM_{q,\varphi}^{\{x_{0}\}}}  &  =\sup_{r>0}\varphi \left(
x_{0},r\right)  ^{-1}|B_{r}|^{-\frac{1}{q}}\left \Vert T_{\alpha}^{\left(
m\right)  }\left(  \overrightarrow{f}\right)  \right \Vert _{L_{q}\left(
B_{r}\right)  }\\
&  \lesssim \sup_{r>0}\varphi \left(  x_{0},r\right)  ^{-1}\int \limits_{r}%
^{\infty}%
{\displaystyle \prod \limits_{i=1}^{m}}
\left \Vert f_{i}\right \Vert _{L_{p_{i}}\left(  B_{t}\right)  }t^{-\frac{n}{q}%
}\frac{dt}{t}\\
&  \lesssim%
{\displaystyle \prod \limits_{i=1}^{m}}
\left \Vert f_{i}\right \Vert _{LM_{p_{i},\varphi_{i}}^{\{x_{0}\}}}.
\end{align*}
Thus we obtain (\ref{3-1}). Also, for $p_{i}=1\left(  i=1,\ldots,m\right)  $,
the proof of the inequality (\ref{3-2}) is similar and we omit the details
here. Hence the proof is completed.
\end{proof}

Particularly, when $\alpha=0$, we can get the following result for the
$m$-linear {Calder\'{o}n-Zygmund operator by (\ref{10*}) and }Theorem
\ref{TeoA}.

\begin{corollary}
\label{corollary 3}Let $x_{0}\in{\mathbb{R}^{n}}$, $1\leq p_{1},\ldots
,p_{m}<\infty$ with $\frac{1}{p}=%
{\displaystyle \sum \limits_{i=1}^{m}}
\frac{1}{p_{i}}$ and $(\varphi_{1},\ldots,\varphi_{m},\varphi)$ satisfies
condition (\ref{316}). Then the operators $M^{\left(  m\right)  }$ and
$\overline{T}^{\left(  m\right)  }$ $\left(  m\in%
\mathbb{N}
\right)  $ are bounded from product space $LM_{p_{1},\varphi_{1}}^{\{x_{0}%
\}}\times \cdots \times LM_{p_{m},\varphi_{m}}^{\{x_{0}\}}$ to $LM_{p,\varphi
}^{\{x_{0}\}}$ for $p_{i}>1$ $\left(  i=1,\ldots,m\right)  $ and from product
$LM_{p_{1},\varphi_{1}}^{\{x_{0}\}}\times \cdots \times LM_{p_{m},\varphi_{m}%
}^{\{x_{0}\}}$ to $WLM_{p,\varphi}^{\{x_{0}\}}$ for $p_{i}\geq1$ $\left(
i=1,\ldots,m\right)  $.
\end{corollary}

\begin{remark}
Note that, in the case of $m=1$ Theorem \ref{teo9} and Corollary
\ref{corollary 3} have been proved in \cite{BGGS, Gurbuz1, Gurbuz2}.
\end{remark}

If $0<\alpha<mn$ and $\overline{T}_{\alpha}^{\left(  m\right)  }$ $\left(
m\in%
\mathbb{N}
\right)  $ is an $m$-linear fractional integral operator, then the condition
of (\ref{4}) is obviously satisfied by (\ref{1}). We can obtain the following
corollary of Theorem \ref{teo9} by Theorem \ref{Teo-1}:

\begin{corollary}
\label{corollaryA}Let $x_{0}\in{\mathbb{R}^{n}}$, $0<\alpha<mn$ and $1\leq
p_{i}<\frac{mn}{\alpha}$ with $\frac{1}{p}=%
{\displaystyle \sum \limits_{i=1}^{m}}
\frac{1}{p_{i}}$, $\frac{1}{q_{i}}=\frac{1}{p_{i}}-\frac{\alpha}{mn}$ and
$\frac{1}{q}=%
{\displaystyle \sum \limits_{i=1}^{m}}
\frac{1}{q_{i}}=\frac{1}{p}-\frac{\alpha}{n}$and also $(\varphi_{1}%
,\ldots,\varphi_{m},\varphi)$ satisfies condition (\ref{316}). Then the
operators $M_{\alpha}^{\left(  m\right)  }$ and $\overline{T}_{\alpha
}^{\left(  m\right)  }$ $\left(  m\in%
\mathbb{N}
\right)  $ are bounded from product space $LM_{p_{1},\varphi_{1}}^{\{x_{0}%
\}}\times \cdots \times LM_{p_{m},\varphi_{m}}^{\{x_{0}\}}$ to $LM_{q,\varphi
}^{\{x_{0}\}}$ for $p_{i}>1$ $\left(  i=1,\ldots,m\right)  $ and from product
$LM_{p_{1},\varphi_{1}}^{\{x_{0}\}}\times \cdots \times LM_{p_{m},\varphi_{m}%
}^{\{x_{0}\}}$ to $WLM_{q,\varphi}^{\{x_{0}\}}$ for $p_{i}\geq1$ $\left(
i=1,\ldots,m\right)  $.
\end{corollary}

\begin{remark}
Note that, in the case of $m=1$ Corollary \ref{corollaryA} has been proved in
\cite{Gurbuz1, Gurbuz2}.
\end{remark}

\section{Boundedness of the commutators of $m$-linear operators generated by
$m$-linear fractional integral operators and local Campanato functions on the
product spaces $LM_{p_{1},\varphi_{1}}^{\{x_{0}\}}\times \cdots \times
LM_{p_{m},\varphi_{m}}^{\{x_{0}\}}$}

In this section we prove the boundedness of the commutator operator
$T_{\alpha,\overrightarrow{b}}^{\left(  m\right)  }$ $\left(  m\in%
\mathbb{N}
\right)  $ with $\overrightarrow{b}\in LC_{q_{i},\lambda_{i}}^{\left \{
x_{0}\right \}  }({\mathbb{R}^{n}})$ on the product generalized local Morrey
spaces $LM_{p_{1},\varphi_{1}}^{\{x_{0}\}}\times \cdots \times LM_{p_{m}%
,\varphi_{m}}^{\{x_{0}\}}$ by using (\ref{5}) and the following Theorem
\ref{Teo 5}.

Since $BMO({\mathbb{R}^{n}})\subset \bigcap \limits_{q>1}LC_{q}^{\left \{
x_{0}\right \}  }({\mathbb{R}^{n}})$, if we only assume $b\in LC_{q}^{\left \{
x_{0}\right \}  }({\mathbb{R}^{n}})$, or more generally $b\in LC_{q,\lambda
}^{\left \{  x_{0}\right \}  }({\mathbb{R}^{n}})$, then $[b,\overline{T}]$ may
not be a bounded operator on $L_{p}({\mathbb{R}^{n}})$, $1<p<\infty$. However,
it has some boundedness properties on other spaces. As a matter of fact,
Grafakos et al. \cite{GraLiYang} have considered the commutator with $b\in
LC_{q}^{\left \{  x_{0}\right \}  }({\mathbb{R}^{n}})$ on Herz spaces for the
first time (see \cite{FuLinLu} for details). Morever, in \cite{BGGS, FuLinLu,
Gurbuz1, Gurbuz2, Gurbuz3, TaoShi} they have considered the commutators with
$b\in LC_{q,\lambda}^{\left \{  x_{0}\right \}  }({\mathbb{R}^{n}})$.

There are two major reasons for considering the problem of commutators. The
first one is that the boundedness of commutators can produce some
characterizations of function spaces (see \cite{BGGS, Chanillo, Gurbuz1,
Gurbuz2, Gurbuz3, Gurbuz4, Janson, Palus, Shi}). The other one is that the
theory of commutators plays an important role in the study of the regularity
of solutions to elliptic and parabolic PDEs of the second order (see
\cite{ChFraL1, ChFraL2, Pal, Softova}). The boundedness of the commutator has
also been generalized to other contexts and important applications to some
non-linear PDEs have been given by Coifman et al. \cite{CLMS}.

The definition of local Campanato space is as follows.

\begin{definition}
\cite{BGGS, Gurbuz1, Gurbuz4} Let $1\leq q<\infty$ and $0\leq \lambda<\frac
{1}{n}$. A function $f\in L_{q}^{loc}\left(  {\mathbb{R}^{n}}\right)  $ is
said to belong to the $LC_{q,\lambda}^{\left \{  x_{0}\right \}  }\left(
{\mathbb{R}^{n}}\right)  $ (local Campanato space), if%
\begin{equation}
\left \Vert f\right \Vert _{LC_{q,\lambda}^{\left \{  x_{0}\right \}  }}%
=\sup_{r>0}\left(  \frac{1}{\left \vert B\left(  x_{0},r\right)  \right \vert
^{1+\lambda q}}\int \limits_{B\left(  x_{0},r\right)  }\left \vert f\left(
y\right)  -f_{B\left(  x_{0},r\right)  }\right \vert ^{q}dy\right)  ^{\frac
{1}{q}}<\infty, \label{e51}%
\end{equation}
where%
\[
f_{B\left(  x_{0},r\right)  }=\frac{1}{\left \vert B\left(  x_{0},r\right)
\right \vert }\int \limits_{B\left(  x_{0},r\right)  }f\left(  y\right)  dy.
\]
Define%
\[
LC_{q,\lambda}^{\left \{  x_{0}\right \}  }\left(  {\mathbb{R}^{n}}\right)
=\left \{  f\in L_{q}^{loc}\left(  {\mathbb{R}^{n}}\right)  :\left \Vert
f\right \Vert _{LC_{q,\lambda}^{\left \{  x_{0}\right \}  }}<\infty \right \}  .
\]

\end{definition}

\begin{remark}
If two functions which differ by a constant are regarded as a function in the
space $LC_{q,\lambda}^{\left \{  x_{0}\right \}  }\left(  {\mathbb{R}^{n}%
}\right)  $, then $LC_{q,\lambda}^{\left \{  x_{0}\right \}  }\left(
{\mathbb{R}^{n}}\right)  $ becomes a Banach space. The space $LC_{q,\lambda
}^{\left \{  x_{0}\right \}  }\left(  {\mathbb{R}^{n}}\right)  $ when
$\lambda=0$ is just the $LC_{q}^{\left \{  x_{0}\right \}  }({\mathbb{R}^{n}})$.
Apparently, (\ref{e51}) is equivalent to the following condition:%
\[
\sup_{r>0}\inf_{c\in%
\mathbb{C}
}\left(  \frac{1}{\left \vert B\left(  x_{0},r\right)  \right \vert ^{1+\lambda
q}}\int \limits_{B\left(  x_{0},r\right)  }\left \vert f\left(  y\right)
-c\right \vert ^{q}dy\right)  ^{\frac{1}{q}}<\infty.
\]

\end{remark}

In \cite{LuYang1}, Lu and Yang have introduced the central BMO space
$CBMO_{q}({\mathbb{R}^{n}})=LC_{q,0}^{\{0\}}({\mathbb{R}^{n}})$. Also the
space $CBMO^{\{x_{0}\}}({\mathbb{R}^{n}})=LC_{1,0}^{\{x_{0}\}}({\mathbb{R}%
^{n}})$ has been considered in other denotes in \cite{Rzaev}. The space
$LC_{q}^{\left \{  x_{0}\right \}  }({\mathbb{R}^{n}})$ can be regarded as a
local version of $BMO({\mathbb{R}^{n}})$, the space of bounded mean
oscillation, at the origin. But, they have quite different properties. The
classical John-Nirenberg inequality shows that functions in $BMO({\mathbb{R}%
^{n}})$ are locally exponentially integrable. This implies that, for any
$1\leq q<\infty$, the functions in $BMO({\mathbb{R}^{n}})$ can be described by
means of the condition: \
\[
\sup_{B\subset{\mathbb{R}^{n}}}\left(  \frac{1}{|B|}%
{\displaystyle \int \limits_{B}}
|f(y)-f_{B}|^{q}dy\right)  ^{1/q}<\infty,
\]
where $B$ denotes an arbitrary ball on ${\mathbb{R}^{n}}$. However, the space
$LC_{q}^{\left \{  x_{0}\right \}  }({\mathbb{R}^{n}})$ depends on $q$. If
$q_{1}<q_{2}$, then $LC_{q_{2}}^{\left \{  x_{0}\right \}  }({\mathbb{R}^{n}%
})\subsetneqq LC_{q_{1}}^{\left \{  x_{0}\right \}  }({\mathbb{R}^{n}})$.
Therefore, there is no analogy of the famous John-Nirenberg inequality of
$BMO({\mathbb{R}^{n}})$ for the space $LC_{q}^{\left \{  x_{0}\right \}
}({\mathbb{R}^{n}})$. One can imagine that the behavior of $LC_{q}^{\left \{
x_{0}\right \}  }({\mathbb{R}^{n}})$ may be quite different from that of
$BMO({\mathbb{R}^{n}})$ (see \cite{LuWu} for details).

\begin{lemma}
\label{Lemma 4}\cite{BGGS, Gurbuz1, Gurbuz2} Let $b$ be function in
$LC_{q,\lambda}^{\left \{  x_{0}\right \}  }\left(
\mathbb{R}
^{n}\right)  $, $1\leq q<\infty$, $0\leq \lambda<\frac{1}{n}$ and $r_{1}$,
$r_{2}>0$. Then%
\begin{equation}
\left(  \frac{1}{\left \vert B\left(  x_{0},r_{1}\right)  \right \vert
^{1+\lambda q}}%
{\displaystyle \int \limits_{B\left(  x_{0},r_{1}\right)  }}
\left \vert b\left(  y\right)  -b_{B\left(  x_{0},r_{2}\right)  }\right \vert
^{q}dy\right)  ^{\frac{1}{q}}\leq C\left(  1+\left \vert \ln \frac{r_{1}}{r_{2}%
}\right \vert \right)  \left \Vert b\right \Vert _{LC_{q,\lambda}^{\left \{
x_{0}\right \}  }}, \label{a*}%
\end{equation}
where $C>0$ is independent of $b$, $r_{1}$ and $r_{2}$.

From this inequality (\ref{a*}), we have for $0<r_{2}<r_{1}$%
\begin{equation}
\left \vert b_{B\left(  x_{0},r_{1}\right)  }-b_{B\left(  x_{0},r_{2}\right)
}\right \vert \leq C\left(  1+\ln \frac{r_{1}}{r_{2}}\right)  \left \vert
B\left(  x_{0},r_{1}\right)  \right \vert ^{\lambda}\left \Vert b\right \Vert
_{LC_{q,\lambda}^{\left \{  x_{0}\right \}  }} \label{b*}%
\end{equation}
and it is easy to see that%
\begin{equation}
\left \Vert b-\left(  b\right)  _{B\left(  x_{0},r\right)  }\right \Vert
_{L_{q}\left(  B\left(  x_{0},r\right)  \right)  }\leq Cr^{\frac{n}%
{q}+n\lambda}\left \Vert b\right \Vert _{LC_{q,\lambda}^{\left \{  x_{0}\right \}
}}. \label{c*}%
\end{equation}

\end{lemma}

Note that one gets the proof of (\ref{a*}) in a similar way as in Lemma 2.1 of
\cite{LinLu}.

In \cite{DingYZ} the boundedness of commutator operators satisfying condition
(\ref{4}) (by taking $m=1$ there) has been proved.

About the commutator of multilinear operators $T_{\alpha}^{\left(  m\right)  }
$ satisfying condition (\ref{4}), we get the following corresponding theorem.

\begin{theorem}
\label{Teo 5}Let $x_{0}\in{\mathbb{R}^{n}}$, $0<\alpha<mn,$ and $1\leq
p_{i}<\frac{mn}{\alpha}$ with $\frac{1}{q}=\sum \limits_{i=1}^{m}\frac{1}%
{p_{i}}+\sum \limits_{i=1}^{m}\frac{1}{q_{i}}-\frac{\alpha}{n}$ and
$\overrightarrow{b}\in LC_{q_{i},\lambda_{i}}^{\left \{  x_{0}\right \}
}({\mathbb{R}^{n}})$ for $0\leq \lambda_{i}<\frac{1}{n}$, $i=1,\ldots,m$. Let
also, $T_{\alpha}^{\left(  m\right)  }$ $\left(  m\in%
\mathbb{N}
\right)  $ be a multilinear operator satisfying condition (\ref{4}), bounded
from $L_{p_{1}}\times \cdots \times L_{p_{m}}$ into $L_{q}$. Then the
inequality
\begin{align}
\Vert T_{\alpha,\overrightarrow{b}}^{\left(  m\right)  }\left(
\overrightarrow{f}\right)  \Vert_{L_{q}(B_{r})}  &  \lesssim%
{\displaystyle \prod \limits_{i=1}^{m}}
\Vert \overrightarrow{b}\Vert_{LC_{q_{i},\lambda_{i}}^{\left \{  x_{0}\right \}
}}r^{\frac{n}{q}}\nonumber \\
&  \times \int \limits_{2r}^{\infty}\left(  1+\ln \frac{t}{r}\right)
^{m}t^{n\left(  -\frac{1}{q}+%
{\displaystyle \sum \limits_{i=1}^{m}}
\lambda_{i}+%
{\displaystyle \sum \limits_{i=1}^{m}}
\frac{1}{q_{i}}\right)  -1}%
{\displaystyle \prod \limits_{i=1}^{m}}
\Vert f_{i}\Vert_{L_{p_{i}}(B_{t})}dt \label{200}%
\end{align}
holds for any ball $B_{r}=B(x_{0},r)$ and for all $\overrightarrow{f}\in
L_{p_{1}}^{loc}\left(  {\mathbb{R}^{n}}\right)  \times \cdots \times L_{p_{m}%
}^{loc}\left(  {\mathbb{R}^{n}}\right)  $.
\end{theorem}

\begin{proof}
As in the proof of Theorem \ref{Teo2}, we consider only the situation when
$m=2$. Actually, a similar procedure works for all $m\in%
\mathbb{N}
$. Thus, without loss of generality, it is sufficient to show that the
conclusion holds for $T_{\alpha,\overrightarrow{b}}^{\left(  2\right)
}\left(  \overrightarrow{f}\right)  =T_{\alpha,\left(  b_{1},b_{2}\right)
}^{\left(  2\right)  }\left(  f_{1},f_{2}\right)  $.

We just consider the case $p_{i},q_{i}>1$ for $i=1,2$. For any $x_{0}%
\in{\mathbb{R}^{n}}$, set $B_{r}=B\left(  x_{0},r\right)  $ for the ball
centered at $x_{0}$ and of radius $r$ and $B_{2r}=B\left(  x_{0},2r\right)  $.
Thus, we have the following decomposition,%
\begin{align*}
T_{\alpha,\left(  b_{1},b_{2}\right)  }^{\left(  2\right)  }\left(
f_{1},f_{2}\right)  \left(  x\right)   &  =\left(  b_{1}-\left(  b_{1}\right)
_{B_{r}}\right)  \left(  b_{2}-\left(  b_{2}\right)  _{B_{r}}\right)
T_{\alpha}^{\left(  2\right)  }\left(  f_{1},f_{2}\right)  \left(  x\right) \\
&  -\left(  b_{1}-\left(  b_{1}\right)  _{B_{r}}\right)  T_{\alpha}^{\left(
2\right)  }\left[  f_{1},\left(  b_{2}\left(  \cdot \right)  -\left(
b_{2}\right)  _{B_{r}}\right)  f_{2}\right]  \left(  x\right) \\
&  -\left(  b_{2}-\left(  b_{2}\right)  _{B_{r}}\right)  T_{\alpha}^{\left(
2\right)  }\left[  \left(  b_{1}\left(  \cdot \right)  -\left(  b_{1}\right)
_{B_{r}}\right)  f_{1},f_{2}\right]  \left(  x\right) \\
&  +T_{\alpha}^{\left(  2\right)  }\left[  \left(  b_{1}\left(  \cdot \right)
-\left(  b_{1}\right)  _{B_{r}}\right)  f_{1},\left(  b_{2}\left(
\cdot \right)  -\left(  b_{2}\right)  _{B_{r}}\right)  f_{2}\right]  \left(
x\right) \\
&  \equiv H_{1}\left(  x\right)  +H_{2}\left(  x\right)  +H_{3}\left(
x\right)  +H_{4}\left(  x\right)  .
\end{align*}
Thus,%
\begin{align}
\left \Vert T_{\alpha,\left(  b_{1},b_{2}\right)  }^{\left(  2\right)  }\left(
f_{1},f_{2}\right)  \right \Vert _{L_{q}\left(  B_{r}\right)  }  &  =\left(
{\displaystyle \int \limits_{B_{r}}}
\left \vert T_{\alpha,\left(  b_{1},b_{2}\right)  }^{\left(  2\right)  }\left(
f_{1},f_{2}\right)  \left(  x\right)  \right \vert ^{q}dx\right)  ^{\frac{1}%
{q}}\nonumber \\
&  \leq%
{\displaystyle \sum \limits_{i=1}^{4}}
\left(
{\displaystyle \int \limits_{B_{r}}}
\left \vert H_{i}\left(  x\right)  \right \vert ^{q}dx\right)  ^{\frac{1}{q}}=%
{\displaystyle \sum \limits_{i=1}^{4}}
G_{i}. \label{100}%
\end{align}

One observes that the estimate of $G_{2}$ is analogous to that of $G_{3}$.
Thus, we will only estimate $G_{1}$, $G_{2}$ and $G_{4}$.

Indeed, we also decompose $f_{i}$ as $f_{i}\left(  y_{i}\right)  =$
$f_{i}\left(  y_{i}\right)  \chi_{B_{2r}}+f_{i}\left(  y_{i}\right)
\chi_{\left(  B_{2r}\right)  ^{c}}$ for $i=1,2$. And, we write $f_{1}%
=f_{1}^{0}+f_{1}^{\infty}$ and $f_{2}=f_{2}^{0}+f_{2}^{\infty}$, where
\ $f_{i}^{0}=f_{i}\chi_{B_{2r}}$, $f_{i}^{\infty}=f_{i}\chi_{\left(
B_{2r}\right)  ^{c}}$, for $i=1,2$.

$\left(  i\right)  $ For $G_{1}=\left \Vert \left(  b_{1}-\left(  b_{1}\right)
_{B_{r}}\right)  \left(  b_{2}-\left(  b_{2}\right)  _{B_{r}}\right)
T_{\alpha,\left(  b_{1},b_{2}\right)  }^{\left(  2\right)  }\left(  f_{1}%
^{0},f_{2}^{0}\right)  \right \Vert _{L_{q}\left(  B_{r}\right)  }$, we
decompose it into four parts as follows:%
\begin{align*}
G_{1}  &  \lesssim \left \Vert \left(  b_{1}-\left(  b_{1}\right)  _{B_{r}%
}\right)  \left(  b_{2}-\left(  b_{2}\right)  _{B_{r}}\right)  T_{\alpha
}^{\left(  2\right)  }\left(  f_{1}^{0},f_{2}^{0}\right)  \right \Vert
_{L_{q}\left(  B_{r}\right)  }\\
&  +\left \Vert \left(  b_{1}-\left(  b_{1}\right)  _{B_{r}}\right)  T_{\alpha
}^{\left(  2\right)  }\left[  f_{1}^{0},\left(  b_{2}-\left(  b_{2}\right)
_{B_{r}}\right)  f_{2}^{0}\right]  \right \Vert _{L_{q}\left(  B_{r}\right)
}\\
&  +\left \Vert \left(  b_{2}-\left(  b_{2}\right)  _{B_{r}}\right)  T_{\alpha
}^{\left(  2\right)  }\left[  \left(  b_{1}-\left(  b_{1}\right)  _{B_{r}%
}\right)  f_{1}^{0},f_{2}^{0}\right]  \right \Vert _{L_{q}\left(  B_{r}\right)
}\\
&  +\left \Vert T_{\alpha}^{\left(  2\right)  }\left[  \left(  b_{1}-\left(
b_{1}\right)  _{B_{r}}\right)  f_{1}^{0},\left(  b_{2}-\left(  b_{2}\right)
_{B_{r}}\right)  f_{2}^{0}\right]  \right \Vert _{L_{q}\left(  B_{r}\right)
}\\
&  \equiv G_{11}+G_{12}+G_{13}+G_{14}.
\end{align*}

Firstly, $1<\overline{p},\overline{q}<\infty$, such that $\frac{1}%
{\overline{p}}=\frac{1}{p_{1}}+\frac{1}{p_{2}}$ and $\frac{1}{\overline{q}%
}=\frac{1}{\overline{p}}-\frac{\alpha}{n},\frac{1}{q}=\frac{1}{\overline{r}%
}+\frac{1}{\overline{q}},\frac{1}{\overline{r}}=\frac{1}{q_{1}}+\frac{1}%
{q_{2}} $. Then, using H\"{o}lder's inequality and from the boundedness of
$T_{\alpha}^{\left(  2\right)  }$ from $L_{p_{1}}\times L_{p_{2}}$ into
$L_{\overline{q}}$ it follows that:%
\begin{align*}
G_{11}  &  \lesssim \left \Vert \left(  b_{1}-\left(  b_{1}\right)  _{B_{r}%
}\right)  \left(  b_{2}-\left(  b_{2}\right)  _{B_{r}}\right)  \right \Vert
_{L_{\overline{r}}\left(  B_{r}\right)  }\left \Vert T_{\alpha}^{\left(
2\right)  }\left(  f_{1}^{0},f_{2}^{0}\right)  \right \Vert _{L_{\overline{q}%
}\left(  B_{r}\right)  }\\
&  \lesssim \left \Vert b_{1}-\left(  b_{1}\right)  _{B}\right \Vert _{L_{q_{1}%
}\left(  B_{r}\right)  }\left \Vert b_{2}-\left(  b_{2}\right)  _{B}\right \Vert
_{L_{q_{2}}\left(  B_{r}\right)  }\left \Vert f_{1}\right \Vert _{L_{p_{1}%
}\left(  B_{2r}\right)  }\left \Vert f_{2}\right \Vert _{L_{p_{2}}\left(
B_{2r}\right)  }\\
&  \lesssim \left \Vert b_{1}-\left(  b_{1}\right)  _{B_{r}}\right \Vert
_{L_{q_{1}}\left(  B_{r}\right)  }\left \Vert b_{2}-\left(  b_{2}\right)
_{B_{r}}\right \Vert _{L_{q_{2}}\left(  B_{r}\right)  }r^{n\left(  \frac
{1}{p_{1}}+\frac{1}{p_{2}}-\frac{\alpha}{n}\right)  }\\
&  \times \int \limits_{2r}^{\infty}\prod \limits_{i=1}^{2}\left \Vert
f_{i}\right \Vert _{L_{p_{i}}\left(  B_{t}\right)  }\frac{dt}{t^{n\left(
\frac{1}{p_{1}}+\frac{1}{p_{2}}\right)  +1-\alpha}}\\
&  \lesssim \Vert b_{1}\Vert_{LC_{q_{1},\lambda_{1}}^{\left \{  x_{0}\right \}
}}\Vert b_{2}\Vert_{LC_{q_{2},\lambda_{2}}^{\left \{  x_{0}\right \}  }%
}r^{n\left(  \frac{1}{q_{1}}+\frac{1}{q_{2}}+\frac{1}{p_{1}}+\frac{1}{p_{2}%
}-\frac{\alpha}{n}\right)  }\\
&  \times \int \limits_{2r}^{\infty}\left(  1+\ln \frac{t}{r}\right)
^{2}t^{n\left(  \lambda_{1}+\lambda_{2}\right)  -n\left(  \frac{1}{p_{1}%
}+\frac{1}{p_{2}}\right)  -1+\alpha}\prod \limits_{i=1}^{2}\left \Vert
f_{i}\right \Vert _{L_{p_{i}}\left(  B_{t}\right)  }dt\\
&  \lesssim \Vert b_{1}\Vert_{LC_{q_{1},\lambda_{1}}^{\left \{  x_{0}\right \}
}}\Vert b_{2}\Vert_{LC_{q_{2},\lambda_{2}}^{\left \{  x_{0}\right \}  }}%
r^{\frac{n}{q}}\\
&  \times \int \limits_{2r}^{\infty}\left(  1+\ln \frac{t}{r}\right)
^{2}t^{n\left(  \lambda_{1}+\lambda_{2}\right)  -\frac{n}{q}+n\left(  \frac
{1}{q_{1}}+\frac{1}{q_{2}}\right)  -1}\prod \limits_{i=1}^{2}\left \Vert
f_{i}\right \Vert _{L_{p_{i}}\left(  B_{t}\right)  }dt.
\end{align*}

Secondly, for $G_{12}$, let $1<\tau<\infty$, such that $\frac{1}{q}=\frac
{1}{q_{1}}+\frac{1}{\tau}$. Then similar to the estimates for $G_{11}$, we
have%
\begin{align*}
G_{12}  &  \lesssim \left \Vert b_{1}-\left(  b_{1}\right)  _{B_{r}}\right \Vert
_{L_{q_{1}}\left(  B_{r}\right)  }\left \Vert T_{\alpha}^{\left(  2\right)  }
\left[  f_{1}^{0},\left(  b_{2}-\left(  b_{2}\right)  _{B_{r}}\right)
f_{2}^{0}\right]  \right \Vert _{L_{\tau}\left(  B_{r}\right)  }\\
&  \lesssim \left \Vert b_{1}-\left(  b_{1}\right)  _{B_{r}}\right \Vert
_{L_{q_{1}}\left(  B_{r}\right)  }\left \Vert f_{1}^{0}\right \Vert
_{_{L_{p_{1}}\left(
\mathbb{R}
^{n}\right)  }}\left \Vert \left(  b_{2}-\left(  b_{2}\right)  _{B_{r}}\right)
f_{2}^{0}\right \Vert _{L_{k}\left(
\mathbb{R}
^{n}\right)  }\\
&  \lesssim \left \Vert b_{1}-\left(  b_{1}\right)  _{B_{r}}\right \Vert
_{L_{q_{1}}\left(  B_{r}\right)  }\left \Vert b_{2}-\left(  b_{2}\right)
_{B_{r}}\right \Vert _{L_{q_{2}}\left(  B_{2r}\right)  }\left \Vert
f_{1}\right \Vert _{L_{p_{1}}\left(  B_{2r}\right)  }\left \Vert f_{2}%
\right \Vert _{L_{p_{2}}\left(  B_{2r}\right)  },
\end{align*}
where $1<k<\frac{2n}{\alpha}$, such that $\frac{1}{k}=\frac{1}{p_{2}}+\frac
{1}{q_{2}}=\frac{1}{\tau}-\frac{1}{p_{1}}+\frac{\alpha}{n}$. From Lemma
\ref{Lemma 4}, it is easy to see that%
\[
\left \Vert b_{i}-\left(  b_{i}\right)  _{B_{r}}\right \Vert _{L_{q_{i}}\left(
B_{r}\right)  }\leq Cr^{\frac{n}{q_{i}}+n\lambda_{i}}\left \Vert b_{i}%
\right \Vert _{LC_{q_{i},\lambda_{i}}^{\left \{  x_{0}\right \}  }},
\]
and%
\begin{align}
\left \Vert b_{i}-\left(  b_{i}\right)  _{B_{r}}\right \Vert _{L_{q_{i}}\left(
B_{2r}\right)  }  &  \leq \left \Vert b_{i}-\left(  b_{i}\right)  _{B_{2r}%
}\right \Vert _{L_{q_{i}}\left(  B_{2r}\right)  }+\left \Vert \left(
b_{i}\right)  _{B_{r}}-\left(  b_{i}\right)  _{B_{2r}}\right \Vert _{L_{q_{i}%
}\left(  B_{2r}\right)  }\nonumber \\
&  \lesssim r^{\frac{n}{q_{i}}+n\lambda_{i}}\left \Vert b_{i}\right \Vert
_{LC_{q_{i},\lambda_{i}}^{\left \{  x_{0}\right \}  }}, \label{*}%
\end{align}
for $i=1$, $2$. Hence, we get%
\begin{align*}
G_{12}  &  \lesssim \Vert b_{1}\Vert_{LC_{q_{1},\lambda_{1}}^{\left \{
x_{0}\right \}  }}\Vert b_{2}\Vert_{LC_{q_{2},\lambda_{2}}^{\left \{
x_{0}\right \}  }}r^{\frac{n}{q}}\\
&  \times \int \limits_{2r}^{\infty}\left(  1+\ln \frac{t}{r}\right)
^{2}t^{n\left(  \lambda_{1}+\lambda_{2}\right)  -\frac{n}{q}+n\left(  \frac
{1}{q_{1}}+\frac{1}{q_{2}}\right)  -1}\prod \limits_{i=1}^{2}\left \Vert
f_{i}\right \Vert _{L_{p_{i}}\left(  B_{t}\right)  }dt.
\end{align*}

Similarly, $G_{13}$ has the same estimate as above, here we omit the details,
thus the inequality%
\begin{align*}
G_{13}  &  \lesssim \Vert b_{1}\Vert_{LC_{q_{1},\lambda_{1}}^{\left \{
x_{0}\right \}  }}\Vert b_{2}\Vert_{LC_{q_{2},\lambda_{2}}^{\left \{
x_{0}\right \}  }}r^{\frac{n}{q}}\\
&  \times \int \limits_{2r}^{\infty}\left(  1+\ln \frac{t}{r}\right)
^{2}t^{n\left(  \lambda_{1}+\lambda_{2}\right)  -\frac{n}{q}+n\left(  \frac
{1}{q_{1}}+\frac{1}{q_{2}}\right)  -1}\prod \limits_{i=1}^{2}\left \Vert
f_{i}\right \Vert _{L_{p_{i}}\left(  B_{t}\right)  }dt
\end{align*}
is valid.

At last, we consider the term $G_{14}$. Let $1<\tau_{1},\tau_{2}<\frac
{2n}{\alpha}$, such that $\frac{1}{\tau_{1}}=\frac{1}{p_{1}}+\frac{1}{q_{1}},$
$\frac{1}{\tau_{2}}=\frac{1}{p_{2}}+\frac{1}{q_{2}}$ and $\frac{1}{q}=\frac
{1}{\tau_{1}}+\frac{1}{\tau_{2}}-\frac{\alpha}{n}$. Then by the boundedness of
$T_{\alpha}^{\left(  2\right)  }$ from $L_{\tau_{1}}\times L_{\tau_{2}}$ into
$L_{q}$, H\"{o}lder's inequality and (\ref{*}), we obtain%
\begin{align*}
G_{14}  &  \lesssim \left \Vert \left(  b_{1}-\left(  b_{1}\right)  _{B_{r}%
}\right)  f_{1}^{0}\right \Vert _{L_{\tau_{1}}\left(  B_{r}\right)  }\left \Vert
\left(  b_{2}-\left(  b_{2}\right)  _{B_{r}}\right)  f_{2}^{0}\right \Vert
_{L_{\tau_{2}}\left(  B_{r}\right)  }\\
&  \lesssim \left \Vert b_{1}-\left(  b_{1}\right)  _{B_{r}}\right \Vert
_{L_{q_{1}}\left(  B_{2r}\right)  }\left \Vert b_{2}-\left(  b_{2}\right)
_{B_{r}}\right \Vert _{L_{q_{2}}\left(  B_{2r}\right)  }\left \Vert
f_{1}\right \Vert _{L_{p_{1}}\left(  B_{2r}\right)  }\left \Vert f_{2}%
\right \Vert _{L_{p_{2}}\left(  B_{2r}\right)  }\\
&  \lesssim \Vert b_{1}\Vert_{LC_{q_{1},\lambda_{1}}^{\left \{  x_{0}\right \}
}}\Vert b_{2}\Vert_{LC_{q_{2},\lambda_{2}}^{\left \{  x_{0}\right \}  }}%
r^{\frac{n}{q}}\\
&  \times \int \limits_{2r}^{\infty}\left(  1+\ln \frac{t}{r}\right)
^{2}t^{n\left(  \lambda_{1}+\lambda_{2}\right)  -\frac{n}{q}+n\left(  \frac
{1}{q_{1}}+\frac{1}{q_{2}}\right)  -1}\prod \limits_{i=1}^{2}\left \Vert
f_{i}\right \Vert _{L_{p_{i}}\left(  B_{t}\right)  }dt.
\end{align*}

Combining all the estimates of $G_{11}$, $G_{12}$, $G_{13}$, $G_{14}$; we get%
\begin{align*}
G_{1}  &  =\left \Vert T_{\alpha,\left(  b_{1},b_{2}\right)  }^{\left(
2\right)  }\left(  f_{1}^{0},f_{2}^{0}\right)  \right \Vert _{L_{q}\left(
B_{r}\right)  }\lesssim \Vert b_{1}\Vert_{LC_{q_{1},\lambda_{1}}^{\left \{
x_{0}\right \}  }}\Vert b_{2}\Vert_{LC_{q_{2},\lambda_{2}}^{\left \{
x_{0}\right \}  }}r^{\frac{n}{q}}\\
&  \times \int \limits_{2r}^{\infty}\left(  1+\ln \frac{t}{r}\right)
^{2}t^{n\left(  \lambda_{1}+\lambda_{2}\right)  -\frac{n}{q}+n\left(  \frac
{1}{q_{1}}+\frac{1}{q_{2}}\right)  -1}\prod \limits_{i=1}^{2}\left \Vert
f_{i}\right \Vert _{L_{p_{i}}\left(  B_{t}\right)  }dt.
\end{align*}

$\left(  ii\right)  $ For $G_{2}=\left \Vert T_{\alpha,\left(  b_{1}%
,b_{2}\right)  }^{\left(  2\right)  }\left(  f_{1}^{0},f_{2}^{\infty}\right)
\right \Vert _{L_{q}\left(  B_{r}\right)  }$, we also write%
\begin{align*}
G_{2}  &  \lesssim \left \Vert \left(  b_{1}-\left(  b_{1}\right)  _{B_{r}%
}\right)  \left(  b_{2}-\left(  b_{2}\right)  _{B_{r}}\right)  T_{\alpha
}^{\left(  2\right)  }\left(  f_{1}^{0},f_{2}^{\infty}\right)  \right \Vert
_{L_{q}\left(  B_{r}\right)  }\\
&  +\left \Vert \left(  b_{1}-\left(  b_{1}\right)  _{B_{r}}\right)  T_{\alpha
}^{\left(  2\right)  }\left[  f_{1}^{0},\left(  b_{2}-\left(  b_{2}\right)
_{B_{r}}\right)  f_{2}^{\infty}\right]  \right \Vert _{L_{q}\left(
B_{r}\right)  }\\
&  +\left \Vert \left(  b_{2}-\left(  b_{2}\right)  _{B_{r}}\right)  T_{\alpha
}^{\left(  2\right)  }\left[  \left(  b_{1}-\left(  b_{1}\right)  _{B}\right)
f_{1}^{0},f_{2}^{\infty}\right]  \right \Vert _{L_{q}\left(  B_{r}\right)  }\\
&  +\left \Vert T_{\alpha}^{\left(  2\right)  }\left[  \left(  b_{1}-\left(
b_{1}\right)  _{B_{r}}\right)  f_{1}^{0},\left(  b_{2}-\left(  b_{2}\right)
_{B_{r}}\right)  f_{2}^{\infty}\right]  \right \Vert _{L_{q}\left(
B_{r}\right)  }\\
&  \equiv G_{21}+G_{22}+G_{23}+G_{24}.
\end{align*}

Let $1<p_{1},p_{2}<\frac{2n}{\alpha}$, such that $\frac{1}{\overline{p}}%
=\frac{1}{p_{1}}+\frac{1}{p_{2}}$ and $\frac{1}{\overline{q}}=\frac
{1}{\overline{p}}-\frac{\alpha}{n}$. Then, using H\"{o}lder's inequality and
noting that in (\ref{10}) $\frac{1}{q}=\frac{1}{p_{1}}+\frac{1}{p_{2}}%
-\frac{\alpha}{n}$, we have%
\begin{align*}
G_{21}  &  =\left \Vert \left(  b_{1}-\left(  b_{1}\right)  _{B_{r}}\right)
\left(  b_{2}-\left(  b_{2}\right)  _{B_{r}}\right)  T_{\alpha}^{\left(
2\right)  }\left(  f_{1}^{0},f_{2}^{\infty}\right)  \right \Vert _{L_{q}\left(
B_{r}\right)  }\\
&  \lesssim \left \Vert \left(  b_{1}-\left(  b_{1}\right)  _{B_{r}}\right)
\left(  b_{2}-\left(  b_{2}\right)  _{B_{r}}\right)  \right \Vert
_{L_{\overline{r}}\left(  B_{r}\right)  }\left \Vert T_{\alpha}^{\left(
2\right)  }\left(  f_{1}^{0},f_{2}^{\infty}\right)  \right \Vert _{L_{\overline
{q}}\left(  B_{r}\right)  }\\
&  \lesssim \left \Vert b_{1}-\left(  b_{1}\right)  _{B_{r}}\right \Vert
_{L_{q_{1}}\left(  B_{r}\right)  }\left \Vert b_{2}-\left(  b_{2}\right)
_{B_{r}}\right \Vert _{L_{q_{2}}\left(  B_{r}\right)  }\\
&  \times r^{\frac{n}{\overline{q}}}\int \limits_{2r}^{\infty}\left \Vert
f_{1}\right \Vert _{L_{p_{1}}\left(  B_{t}\right)  }\left \Vert f_{2}\right \Vert
_{L_{p_{2}}\left(  B_{t}\right)  }t^{-n\left(  \frac{1}{p_{1}}+\frac{1}{p_{2}%
}\right)  +\alpha-1}dt\\
&  \lesssim \Vert b_{1}\Vert_{LC_{q_{1},\lambda_{1}}^{\left \{  x_{0}\right \}
}}\Vert b_{2}\Vert_{LC_{q_{2},\lambda_{2}}^{\left \{  x_{0}\right \}  }%
}r^{n\left(  \frac{1}{q_{1}}+\frac{1}{q_{2}}\right)  +n\left(  \lambda
_{1}+\lambda_{2}\right)  }r^{n\left(  \frac{1}{p_{1}}+\frac{1}{p_{2}}%
-\frac{\alpha}{n}\right)  }\\
&  \times \int \limits_{2r}^{\infty}\left(  1+\ln \frac{t}{r}\right)
^{2}t^{-n\left(  \frac{1}{p_{1}}+\frac{1}{p_{2}}\right)  +\alpha-1}%
\prod \limits_{i=1}^{2}\left \Vert f_{i}\right \Vert _{L_{p_{i}}\left(
B_{t}\right)  }dt\\
&  \lesssim \Vert b_{1}\Vert_{LC_{q_{1},\lambda_{1}}^{\left \{  x_{0}\right \}
}}\Vert b_{2}\Vert_{LC_{q_{2},\lambda_{2}}^{\left \{  x_{0}\right \}  }}%
r^{\frac{n}{q}}\\
&  \times \int \limits_{2r}^{\infty}\left(  1+\ln \frac{t}{r}\right)
^{2}t^{n\left(  \lambda_{1}+\lambda_{2}\right)  -\frac{n}{q}+n\left(  \frac
{1}{q_{1}}+\frac{1}{q_{2}}\right)  -1}\prod \limits_{i=1}^{2}\left \Vert
f_{i}\right \Vert _{L_{p_{i}}\left(  B_{t}\right)  }dt
\end{align*}
where $\frac{1}{q}=\frac{1}{\overline{r}}+\frac{1}{\overline{q}},\frac
{1}{\overline{r}}=\frac{1}{q_{1}}+\frac{1}{q_{2}}$.

For estimate $G_{22}$, we use condition (\ref{4}) with $m=2$ and we get%
\begin{align*}
&  \left \vert T_{\alpha}^{\left(  2\right)  }\left[  f_{1}^{0},\left(
b_{2}\left(  \cdot \right)  -\left(  b_{2}\right)  _{B_{r}}\right)
f_{2}^{\infty}\right]  \left(  x\right)  \right \vert \\
&  \lesssim%
{\displaystyle \int \limits_{B_{2r}}}
\left \vert f_{1}\left(  y_{1}\right)  \right \vert dy_{1}%
{\displaystyle \int \limits_{\left(  B_{2r}\right)  ^{c}}}
\frac{\left \vert b_{2}\left(  y_{2}\right)  -\left(  b_{2}\right)
_{B}\right \vert \left \vert f_{2}\left(  y_{2}\right)  \right \vert }{\left \vert
x_{0}-y_{2}\right \vert ^{2n-\alpha}}dy_{2}.
\end{align*}
It's obvious that%
\begin{equation}%
{\displaystyle \int \limits_{B_{2r}}}
\left \vert f_{1}\left(  y_{1}\right)  \right \vert dy_{1}\lesssim \left \Vert
f_{1}\right \Vert _{L_{p_{1}}\left(  B_{2r}\right)  }\left \vert B_{2r}%
\right \vert ^{1-\frac{1}{p_{1}}}, \label{11}%
\end{equation}
and using H\"{o}lder's inequality and by (\ref{b*}) and (\ref{*}) we have%
\begin{align}
&
{\displaystyle \int \limits_{\left(  B_{2r}\right)  ^{c}}}
\frac{\left \vert b_{2}\left(  y_{2}\right)  -\left(  b_{2}\right)
_{B}\right \vert \left \vert f_{2}\left(  y_{2}\right)  \right \vert }{\left \vert
x_{0}-y_{2}\right \vert ^{2n-\alpha}}dy_{2}\nonumber \\
&  \lesssim%
{\displaystyle \int \limits_{\left(  B_{2r}\right)  ^{c}}}
\left \vert b_{2}\left(  y_{2}\right)  -\left(  b_{2}\right)  _{B_{r}%
}\right \vert \left \vert f_{2}\left(  y_{2}\right)  \right \vert \left[
{\displaystyle \int \limits_{\left \vert x_{0}-y_{2}\right \vert }^{\infty}}
\frac{dt}{t^{2n-\alpha+1}}\right]  dy_{2}\nonumber \\
&  \lesssim%
{\displaystyle \int \limits_{2r}^{\infty}}
\left \Vert b_{2}\left(  y_{2}\right)  -\left(  b_{2}\right)  _{B_{t}%
}\right \Vert _{L_{q_{2}}\left(  B_{t}\right)  }\left \Vert f_{2}\right \Vert
_{L_{p_{2}}\left(  B_{t}\right)  }\left \vert B_{t}\right \vert ^{1-\left(
\frac{1}{p_{2}}+\frac{1}{q_{2}}\right)  }\frac{dt}{t^{2n-\alpha+1}}\nonumber \\
&  +%
{\displaystyle \int \limits_{2r}^{\infty}}
\left \vert \left(  b_{2}\right)  _{B_{t}}-\left(  b_{2}\right)  _{B_{r}%
}\right \vert \left \Vert f_{2}\right \Vert _{L_{p_{2}}\left(  B_{t}\right)
}\left \vert B_{t}\right \vert ^{1-\frac{1}{p_{2}}}\frac{dt}{t^{2n-\alpha+1}%
}\nonumber \\
&  \lesssim \Vert b_{2}\Vert_{LC_{q_{2},\lambda_{2}}^{\left \{  x_{0}\right \}
}}%
{\displaystyle \int \limits_{2r}^{\infty}}
\left \vert B_{t}\right \vert ^{\frac{1}{q_{2}}+\lambda_{2}}\left \Vert
f_{2}\right \Vert _{L_{p_{2}}\left(  B_{t}\right)  }\left \vert B_{t}\right \vert
^{1-\left(  \frac{1}{p_{2}}+\frac{1}{q_{2}}\right)  }\frac{dt}{t^{2n-\alpha
+1}}\nonumber \\
&  +\Vert b_{2}\Vert_{LC_{q_{2},\lambda_{2}}^{\left \{  x_{0}\right \}  }}%
\int \limits_{2r}^{\infty}\left(  1+\ln \frac{t}{r}\right)  \left \vert
B_{t}\right \vert ^{\lambda_{2}}\left \Vert f_{2}\right \Vert _{L_{p_{2}}\left(
B_{t}\right)  }\left \vert B_{t}\right \vert ^{1-\frac{1}{p_{2}}}\frac
{dt}{t^{2n-\alpha+1}}\nonumber \\
&  \lesssim \Vert b_{2}\Vert_{LC_{q_{2},\lambda_{2}}^{\left \{  x_{0}\right \}
}}\int \limits_{2r}^{\infty}\left(  1+\ln \frac{t}{r}\right)  ^{2}%
t^{-n+n\lambda_{2}-\frac{n}{p_{2}}-1+\alpha}\left \Vert f_{2}\right \Vert
_{L_{p_{2}}\left(  B_{t}\right)  }dt. \label{11*}%
\end{align}
Hence, by (\ref{11}) and (\ref{11*}), it follows that:%
\begin{align*}
&  \left \vert T_{\alpha}^{\left(  2\right)  }\left[  f_{1}^{0},\left(
b_{2}\left(  \cdot \right)  -\left(  b_{2}\right)  _{B_{r}}\right)
f_{2}^{\infty}\right]  \left(  x\right)  \right \vert \\
&  \lesssim \Vert b_{2}\Vert_{LC_{q_{2},\lambda_{2}}^{\left \{  x_{0}\right \}
}}\left \Vert f_{1}\right \Vert _{L_{p_{1}}\left(  B_{2r}\right)  }\left \vert
B_{2r}\right \vert ^{1-\frac{1}{p_{1}}}\int \limits_{2r}^{\infty}\left(
1+\ln \frac{t}{r}\right)  ^{2}t^{-n+n\lambda_{2}-\frac{n}{p_{2}}-1+\alpha
}\left \Vert f_{2}\right \Vert _{L_{p_{2}}\left(  B_{t}\right)  }dt\\
&  \lesssim \Vert b_{2}\Vert_{LC_{q_{2},\lambda_{2}}^{\left \{  x_{0}\right \}
}}\int \limits_{2r}^{\infty}\left(  1+\ln \frac{t}{r}\right)  ^{2}%
t^{n\lambda_{2}-n\left(  \frac{1}{p_{1}}+\frac{1}{p_{2}}\right)  -1+\alpha
}\prod \limits_{i=1}^{2}\left \Vert f_{i}\right \Vert _{L_{p_{i}}\left(
B_{t}\right)  }dt.
\end{align*}
Thus, let $1<\tau<\infty$, such that $\frac{1}{q}=\frac{1}{q_{1}}+\frac
{1}{\tau}$. Then similar to the estimates for $G_{11}$, we have%
\begin{align*}
G_{22}  &  =\left \Vert \left(  b_{1}-\left(  b_{1}\right)  _{B_{r}}\right)
T_{\alpha}^{\left(  2\right)  }\left[  f_{1}^{0},\left(  b_{2}-\left(
b_{2}\right)  _{B_{r}}\right)  f_{2}^{\infty}\right]  \right \Vert
_{L_{q}\left(  B_{r}\right)  }\\
&  \lesssim \left \Vert b_{1}-\left(  b_{1}\right)  _{B_{r}}\right \Vert
_{L_{q_{1}}\left(  B_{r}\right)  }\left \Vert T_{\alpha}^{\left(  2\right)  }
\left[  f_{1}^{0},\left(  b_{2}-\left(  b_{2}\right)  _{B_{r}}\right)
f_{2}^{\infty}\right]  \right \Vert _{L_{\tau}\left(  B_{r}\right)  }\\
&  \lesssim \Vert b_{1}\Vert_{LC_{q_{1},\lambda_{1}}^{\left \{  x_{0}\right \}
}}\Vert b_{2}\Vert_{LC_{q_{2},\lambda_{2}}^{\left \{  x_{0}\right \}  }%
}\left \vert B_{r}\right \vert ^{\lambda_{1}+\frac{1}{q_{1}}+\frac{1}{\tau}}\\
&  \times \int \limits_{2r}^{\infty}\left(  1+\ln \frac{t}{r}\right)
^{2}t^{n\lambda_{2}-n\left(  \frac{1}{p_{1}}+\frac{1}{p_{2}}\right)
-1+\alpha}\prod \limits_{i=1}^{2}\left \Vert f_{i}\right \Vert _{L_{p_{i}}\left(
B_{t}\right)  }dt\\
&  \lesssim \Vert b_{1}\Vert_{LC_{q_{1},\lambda_{1}}^{\left \{  x_{0}\right \}
}}\Vert b_{2}\Vert_{LC_{q_{2},\lambda_{2}}^{\left \{  x_{0}\right \}  }}%
r^{\frac{n}{q}}\\
&  \times \int \limits_{2r}^{\infty}\left(  1+\ln \frac{t}{r}\right)
^{2}t^{n\left(  \lambda_{1}+\lambda_{2}\right)  -\frac{n}{q}+n\left(  \frac
{1}{q_{1}}+\frac{1}{q_{2}}\right)  -1}\prod \limits_{i=1}^{2}\left \Vert
f_{i}\right \Vert _{L_{p_{i}}\left(  B_{t}\right)  }dt.
\end{align*}

Similarly, $G_{23}$ has the same estimate above, here we omit the details,
thus the inequality%
\begin{align*}
G_{23}  &  =\left \Vert \left(  b_{2}-\left(  b_{2}\right)  _{B_{r}}\right)
T_{\alpha}^{\left(  2\right)  }\left[  \left(  b_{1}-\left(  b_{1}\right)
_{B_{r}}\right)  f_{1}^{0},f_{2}^{\infty}\right]  \right \Vert _{L_{q}\left(
B_{r}\right)  }\\
&  \lesssim \Vert b_{1}\Vert_{LC_{q_{1},\lambda_{1}}^{\left \{  x_{0}\right \}
}}\Vert b_{2}\Vert_{LC_{q_{2},\lambda_{2}}^{\left \{  x_{0}\right \}  }}%
r^{\frac{n}{q}}\\
&  \times \int \limits_{2r}^{\infty}\left(  1+\ln \frac{t}{r}\right)
^{2}t^{n\left(  \lambda_{1}+\lambda_{2}\right)  -\frac{n}{q}+n\left(  \frac
{1}{q_{1}}+\frac{1}{q_{2}}\right)  -1}\prod \limits_{i=1}^{2}\left \Vert
f_{i}\right \Vert _{L_{p_{i}}\left(  B_{t}\right)  }dt
\end{align*}
is valid.

Now, using the condition (\ref{4}) with $m=2$, we have%
\begin{align*}
&  \left \vert T_{\alpha}^{\left(  2\right)  }\left[  \left(  b_{1}-\left(
b_{1}\right)  _{B_{r}}\right)  f_{1}^{0},\left(  b_{2}-\left(  b_{2}\right)
_{B_{r}}\right)  f_{2}^{\infty}\right]  \left(  x\right)  \right \vert \\
&  \lesssim%
{\displaystyle \int \limits_{B_{2r}}}
\left \vert b_{1}\left(  y_{1}\right)  -\left(  b_{1}\right)  _{B_{r}%
}\right \vert \left \vert f_{1}\left(  y_{1}\right)  \right \vert dy_{1}%
{\displaystyle \int \limits_{\left(  B_{2r}\right)  ^{c}}}
\frac{\left \vert b_{2}\left(  y_{2}\right)  -\left(  b_{2}\right)  _{B_{r}%
}\right \vert \left \vert f_{2}\left(  y_{2}\right)  \right \vert }{\left \vert
x_{0}-y_{2}\right \vert ^{2n-\alpha}}dy_{2}.
\end{align*}
It's obvious that from H\"{o}lder's inequality and (\ref{c*})
\begin{equation}%
{\displaystyle \int \limits_{B_{2r}}}
\left \vert b_{1}\left(  y_{1}\right)  -\left(  b_{1}\right)  _{B_{r}%
}\right \vert \left \vert f_{1}\left(  y_{1}\right)  \right \vert dy_{1}%
\lesssim \Vert b_{1}\Vert_{LC_{q_{1},\lambda_{1}}^{\left \{  x_{0}\right \}  }%
}\left \vert B_{r}\right \vert ^{\lambda_{1}+1-\frac{1}{p_{1}}}\left \Vert
f_{1}\right \Vert _{L_{p_{1}}\left(  B_{2r}\right)  }. \label{12}%
\end{equation}
Then, by (\ref{11*}) and (\ref{12}) we have%
\begin{align*}
&  \left \vert T_{\alpha}^{\left(  2\right)  }\left[  \left(  b_{1}-\left(
b_{1}\right)  _{B_{r}}\right)  f_{1}^{0},\left(  b_{2}-\left(  b_{2}\right)
_{B_{r}}\right)  f_{2}^{\infty}\right]  \left(  x\right)  \right \vert \\
&  \leq \Vert b_{1}\Vert_{LC_{q_{1},\lambda_{1}}^{\left \{  x_{0}\right \}  }%
}\Vert b_{2}\Vert_{LC_{q_{2},\lambda_{2}}^{\left \{  x_{0}\right \}  }}%
\int \limits_{2r}^{\infty}\left(  1+\ln \frac{t}{r}\right)  ^{2}t^{n\left(
\lambda_{1}+\lambda_{2}\right)  -\frac{n}{q}+n\left(  \frac{1}{q_{1}}+\frac
{1}{q_{2}}\right)  -1}\prod \limits_{i=1}^{2}\left \Vert f_{i}\right \Vert
_{L_{p_{i}}\left(  B_{t}\right)  }dt.
\end{align*}
Therefore,%
\begin{align*}
G_{24}  &  =\left \Vert T_{\alpha}^{\left(  2\right)  }\left[  \left(
b_{1}-\left(  b_{1}\right)  _{B_{r}}\right)  f_{1}^{0},\left(  b_{2}-\left(
b_{2}\right)  _{B_{r}}\right)  f_{2}^{\infty}\right]  \right \Vert
_{L_{q}\left(  B_{r}\right)  }\\
&  \lesssim \Vert b_{1}\Vert_{LC_{q_{1},\lambda_{1}}^{\left \{  x_{0}\right \}
}}\Vert b_{2}\Vert_{LC_{q_{2},\lambda_{2}}^{\left \{  x_{0}\right \}  }}%
r^{\frac{n}{q}}\\
&  \times \int \limits_{2r}^{\infty}\left(  1+\ln \frac{t}{r}\right)
^{2}t^{n\left(  \lambda_{1}+\lambda_{2}\right)  -\frac{n}{q}+n\left(  \frac
{1}{q_{1}}+\frac{1}{q_{2}}\right)  -1}\prod \limits_{i=1}^{2}\left \Vert
f_{i}\right \Vert _{L_{p_{i}}\left(  B_{t}\right)  }dt.
\end{align*}

Putting estimates $G_{21},$ $G_{22}$, $G_{23}$, $G_{24}$ together, we get the
desired conclusion%
\begin{align*}
G_{2}  &  =\left \Vert T_{\alpha,\left(  b_{1},b_{2}\right)  }^{\left(
2\right)  }\left(  f_{1}^{0},f_{2}^{\infty}\right)  \right \Vert _{L_{q}\left(
B_{r}\right)  }\lesssim \Vert b_{1}\Vert_{LC_{q_{1},\lambda_{1}}^{\left \{
x_{0}\right \}  }}\Vert b_{2}\Vert_{LC_{q_{2},\lambda_{2}}^{\left \{
x_{0}\right \}  }}r^{\frac{n}{q}}\\
&  \times \int \limits_{2r}^{\infty}\left(  1+\ln \frac{t}{r}\right)
^{2}t^{n\left(  \lambda_{1}+\lambda_{2}\right)  -\frac{n}{q}+n\left(  \frac
{1}{q_{1}}+\frac{1}{q_{2}}\right)  -1}\prod \limits_{i=1}^{2}\left \Vert
f_{i}\right \Vert _{L_{p_{i}}\left(  B_{t}\right)  }dt.
\end{align*}

$(iii)$ Similarly, we have%
\begin{align*}
G_{3}  &  =\left \Vert T_{\alpha,\left(  b_{1},b_{2}\right)  }^{\left(
2\right)  }\left(  f_{1}^{\infty},f_{2}^{0}\right)  \right \Vert _{L_{q}\left(
B_{r}\right)  }\lesssim \Vert b_{1}\Vert_{LC_{q_{1},\lambda_{1}}^{\left \{
x_{0}\right \}  }}\Vert b_{2}\Vert_{LC_{q_{2},\lambda_{2}}^{\left \{
x_{0}\right \}  }}r^{\frac{n}{q}}\\
&  \times \int \limits_{2r}^{\infty}\left(  1+\ln \frac{t}{r}\right)
^{2}t^{n\left(  \lambda_{1}+\lambda_{2}\right)  -\frac{n}{q}+n\left(  \frac
{1}{q_{1}}+\frac{1}{q_{2}}\right)  -1}\prod \limits_{i=1}^{2}\left \Vert
f_{i}\right \Vert _{L_{p_{i}}\left(  B_{t}\right)  }dt.
\end{align*}

$(iv)$ Finally, for $G_{4}=\left \Vert T_{\alpha,\left(  b_{1},b_{2}\right)
}^{\left(  2\right)  }\left(  f_{1}^{\infty},f_{2}^{\infty}\right)
\right \Vert _{L_{q}\left(  B_{r}\right)  }$, we write%
\begin{align*}
G_{4}  &  \lesssim \left \Vert \left(  b_{1}-\left(  b_{1}\right)  _{B_{r}%
}\right)  \left(  b_{2}-\left(  b_{2}\right)  _{B_{r}}\right)  T_{\alpha
}^{\left(  2\right)  }\left(  f_{1}^{\infty},f_{2}^{\infty}\right)
\right \Vert _{L_{q}\left(  B_{r}\right)  }\\
&  +\left \Vert \left(  b_{1}-\left(  b_{1}\right)  _{B_{r}}\right)  T_{\alpha
}^{\left(  2\right)  }\left[  f_{1}^{\infty},\left(  b_{2}-\left(
b_{2}\right)  _{B_{r}}\right)  f_{2}^{\infty}\right]  \right \Vert
_{L_{q}\left(  B_{r}\right)  }\\
&  +\left \Vert \left(  b_{2}-\left(  b_{2}\right)  _{B_{r}}\right)  T_{\alpha
}^{\left(  2\right)  }\left[  \left(  b_{1}-\left(  b_{1}\right)  _{B_{r}%
}\right)  f_{1}^{\infty},f_{2}^{\infty}\right]  \right \Vert _{L_{q}\left(
B_{r}\right)  }\\
&  +\left \Vert T^{\left(  2\right)  }\left[  \left(  b_{1}-\left(
b_{1}\right)  _{B_{r}}\right)  f_{1}^{\infty},\left(  b_{2}-\left(
b_{2}\right)  _{B_{r}}\right)  f_{2}^{\infty}\right]  \right \Vert
_{L_{q}\left(  B_{r}\right)  }\\
&  \equiv G_{41}+G_{42}+G_{43}+G_{44}.
\end{align*}

Let us estimate $G_{41}$, $G_{42}$, $G_{43}$, $G_{44}$ respectively.

Let $1<\tau<\infty$, such that $\frac{1}{q}=\left(  \frac{1}{q_{1}}+\frac
{1}{q_{2}}\right)  +\frac{1}{\tau}$. Then, by H\"{o}lder's inequality and
noting that in (\ref{0}) $\frac{1}{q}=\frac{1}{p_{1}}+\frac{1}{p_{2}}%
-\frac{\alpha}{n}$, we get%
\begin{align*}
G_{41}  &  =\left \Vert \left(  b_{1}-\left(  b_{1}\right)  _{B_{r}}\right)
\left(  b_{2}-\left(  b_{2}\right)  _{B_{r}}\right)  T_{\alpha}^{\left(
2\right)  }\left(  f_{1}^{\infty},f_{2}^{\infty}\right)  \right \Vert
_{L_{q}\left(  B_{r}\right)  }\\
&  \lesssim \left \Vert b_{1}-\left(  b_{1}\right)  _{B_{r}}\right \Vert
_{L_{q_{1}}\left(  B_{r}\right)  }\left \Vert b_{2}-\left(  b_{2}\right)
_{B_{r}}\right \Vert _{L_{q_{2}}\left(  B_{r}\right)  }\left \Vert T_{\alpha
}^{\left(  2\right)  }\left(  f_{1}^{\infty},f_{2}^{\infty}\right)
\right \Vert _{L_{\tau}\left(  B_{r}\right)  }\\
&  \lesssim \Vert b_{1}\Vert_{LC_{q_{1},\lambda_{1}}^{\left \{  x_{0}\right \}
}}\Vert b_{2}\Vert_{LC_{q_{2},\lambda_{2}}^{\left \{  x_{0}\right \}  }%
}\left \vert B_{r}\right \vert ^{\left(  \lambda_{1}+\lambda_{2}\right)
+\left(  \frac{1}{q_{1}}+\frac{1}{q_{2}}\right)  +\frac{1}{\tau}}\\
&  \times \int \limits_{2r}^{\infty}\left \Vert f_{1}\right \Vert _{L_{p_{1}%
}\left(  B_{t}\right)  }\left \Vert f_{2}\right \Vert _{L_{p_{2}}\left(
B_{t}\right)  }t^{-n\left(  \frac{1}{p_{1}}+\frac{1}{p_{2}}\right)  -1+\alpha
}dt\\
&  \lesssim \Vert b_{1}\Vert_{LC_{q_{1},\lambda_{1}}^{\left \{  x_{0}\right \}
}}\Vert b_{2}\Vert_{LC_{q_{2},\lambda_{2}}^{\left \{  x_{0}\right \}  }}%
r^{\frac{n}{q}}\\
&  \times \int \limits_{2r}^{\infty}\left(  1+\ln \frac{t}{r}\right)
^{2}t^{n\left(  \lambda_{1}+\lambda_{2}\right)  -\frac{n}{q}+n\left(  \frac
{1}{q_{1}}+\frac{1}{q_{2}}\right)  -1}\prod \limits_{i=1}^{2}\left \Vert
f_{i}\right \Vert _{L_{p_{i}}\left(  B_{t}\right)  }dt.
\end{align*}

Recalling the estimates used for $G_{22}$, $G_{23}$, $G_{24}$ and also using
the condition (\ref{4}) with $m=2$, we have%
\begin{align*}
&  \left \vert T_{\alpha}^{\left(  2\right)  }\left[  f_{1}^{\infty},\left(
b_{2}-\left(  b_{2}\right)  _{B_{r}}\right)  f_{2}^{\infty}\right]  \left(
x\right)  \right \vert \\
&  \lesssim%
{\displaystyle \int \limits_{\mathbb{R} ^{n}}}
{\displaystyle \int \limits_{\mathbb{R} ^{n}}}
\frac{\left \vert b_{2}\left(  y_{2}\right)  -\left(  b_{2}\right)  _{B_{r}%
}\right \vert \left \vert f_{1}\left(  y_{1}\right)  \chi_{\left(
B_{2r}\right)  ^{c}}\right \vert \left \vert f_{2}\left(  y_{2}\right)
\chi_{\left(  B_{2r}\right)  ^{c}}\right \vert }{\left \vert \left(  x_{0}%
-y_{1},x_{0}-y_{2}\right)  \right \vert ^{2n-\alpha}}dy_{1}dy_{2}\\
&  \lesssim%
{\displaystyle \int \limits_{\left(  B_{2r}\right)  ^{c}}}
{\displaystyle \int \limits_{\left(  B_{2r}\right)  ^{c}}}
\frac{\left \vert b_{2}\left(  y_{2}\right)  -\left(  b_{2}\right)  _{B_{r}%
}\right \vert \left \vert f_{1}\left(  y_{1}\right)  \right \vert \left \vert
f_{2}\left(  y_{2}\right)  \right \vert }{\left \vert x_{0}-y_{1}\right \vert
^{n-\frac{\alpha}{2}}\left \vert x_{0}-y_{2}\right \vert ^{n-\frac{\alpha}{2}}%
}dy_{1}dy_{2}\\
&  \lesssim%
{\displaystyle \sum \limits_{j=1}^{\infty}}
{\displaystyle \int \limits_{B_{2^{j+1}r}\backslash B_{2^{j}r}}}
\frac{\left \vert f_{1}\left(  y_{1}\right)  \right \vert }{\left \vert
x_{0}-y_{1}\right \vert ^{n-\frac{\alpha}{2}}}dy_{1}%
{\displaystyle \int \limits_{B_{2^{j+1}r}\backslash B_{2^{j}r}}}
\frac{\left \vert b_{2}\left(  y_{2}\right)  -\left(  b_{2}\right)  _{B_{r}%
}\right \vert \left \vert f_{2}\left(  y_{2}\right)  \right \vert }{\left \vert
x_{0}-y_{2}\right \vert ^{n-\frac{\alpha}{2}}}dy_{2}\\
&  \lesssim%
{\displaystyle \sum \limits_{j=1}^{\infty}}
\left(  2^{j}r\right)  ^{-2n+\alpha}%
{\displaystyle \int \limits_{B_{2^{j+1}r}}}
\left \vert f_{1}\left(  y_{1}\right)  \right \vert dy_{1}%
{\displaystyle \int \limits_{B_{2^{j+1}r}}}
\left \vert b_{2}\left(  y_{2}\right)  -\left(  b_{2}\right)  _{B_{r}%
}\right \vert \left \vert f_{2}\left(  y_{2}\right)  \right \vert dy_{2}.
\end{align*}
It's obvious that%
\begin{equation}%
{\displaystyle \int \limits_{B_{2^{j+1}r}}}
\left \vert f_{1}\left(  y_{1}\right)  \right \vert dy_{1}\leq \left \Vert
f_{1}\right \Vert _{L_{p_{1}}(B_{2^{j+1}r})}\left \vert B_{2^{j+1}r}\right \vert
^{1-\frac{1}{p_{1}}}, \label{f1}%
\end{equation}
and using H\"{o}lder's inequality and by (\ref{*})%
\begin{align}
&
{\displaystyle \int \limits_{B_{2^{j+1}r}}}
\left \vert b_{2}\left(  y_{2}\right)  -\left(  b_{2}\right)  _{B_{r}%
}\right \vert \left \vert f_{2}\left(  y_{2}\right)  \right \vert dy_{2}%
\nonumber \\
&  \leq \left \Vert b_{2}-\left(  b_{2}\right)  _{B_{2^{j+1}r}}\right \Vert
_{L_{q_{2}}(B_{2^{j+1}r})}\left \Vert f_{2}\right \Vert _{L_{p_{2}}(B_{2^{j+1}%
r})}\left \vert B_{2^{j+1}r}\right \vert ^{1-\left(  \frac{1}{p_{2}}+\frac
{1}{q_{2}}\right)  }\nonumber \\
&  +\left \vert \left(  b_{2}\right)  _{B_{2^{j+1}r}}-\left(  b_{2}\right)
_{B_{r}}\right \vert \left \Vert f_{2}\right \Vert _{L_{p_{2}}(B_{2^{j+1}r}%
)}\left \vert B_{2^{j+1}r}\right \vert ^{1-\frac{1}{p_{2}}}\nonumber \\
&  \lesssim \Vert b_{2}\Vert_{LC_{q_{2},\lambda_{2}}^{\left \{  x_{0}\right \}
}}\left \vert B_{2^{j+1}r}\right \vert ^{\frac{1}{q_{2}}+\lambda_{2}}\left \Vert
f_{2}\right \Vert _{L_{p_{2}}\left(  B_{2^{j+1}r}\right)  }\left \vert
B_{2^{j+1}r}\right \vert ^{1-\left(  \frac{1}{p_{2}}+\frac{1}{q_{2}}\right)
}\nonumber \\
&  +\Vert b_{2}\Vert_{LC_{q_{2},\lambda_{2}}^{\left \{  x_{0}\right \}  }%
}\left(  1+\ln \frac{2^{j+1}r}{r}\right)  \left \vert B_{2^{j+1}r}\right \vert
^{\lambda_{2}}\left \Vert f_{2}\right \Vert _{L_{p_{2}}\left(  B_{2^{j+1}%
r}\right)  }\left \vert B_{2^{j+1}r}\right \vert ^{1-\frac{1}{p_{2}}}\nonumber \\
&  \lesssim \Vert b_{2}\Vert_{LC_{q_{2},\lambda_{2}}^{\left \{  x_{0}\right \}
}}\left(  1+\ln \frac{2^{j+1}r}{r}\right)  ^{2}\left \vert B_{2^{j+1}%
r}\right \vert ^{\lambda_{2}-\frac{1}{p_{2}}+1}\left \Vert f_{2}\right \Vert
_{L_{p_{2}}\left(  B_{2^{j+1}r}\right)  }. \label{f2}%
\end{align}
Hence, by (\ref{f1}) and (\ref{f2}), it follows that:%
\begin{align*}
&  \left \vert T_{\alpha}^{\left(  2\right)  }\left[  f_{1}^{\infty},\left(
b_{2}-\left(  b_{2}\right)  _{B_{r}}\right)  f_{2}^{\infty}\right]  \left(
x\right)  \right \vert \\
&  \lesssim%
{\displaystyle \sum \limits_{j=1}^{\infty}}
\left(  2^{j}r\right)  ^{-2n+\alpha}%
{\displaystyle \int \limits_{B_{2^{j+1}r}}}
\left \vert f_{1}\left(  y_{1}\right)  \right \vert dy_{1}%
{\displaystyle \int \limits_{B_{2^{j+1}r}}}
\left \vert b_{2}\left(  y_{2}\right)  -\left(  b_{2}\right)  _{B_{r}%
}\right \vert \left \vert f_{2}\left(  y_{2}\right)  \right \vert dy_{2}\\
&  \lesssim \Vert b_{2}\Vert_{LC_{q_{2},\lambda_{2}}^{\left \{  x_{0}\right \}
}}%
{\displaystyle \sum \limits_{j=1}^{\infty}}
\left(  2^{j}r\right)  ^{-2n+\alpha}\left(  1+\ln \frac{2^{j+1}r}{r}\right)
^{2}\left \vert B_{2^{j+1}r}\right \vert ^{\lambda_{2}-\left(  \frac{1}{p_{1}%
}+\frac{1}{p_{2}}\right)  +2}\prod \limits_{i=1}^{2}\left \Vert f_{i}\right \Vert
_{L_{p_{i}}\left(  B_{2^{j+1}r}\right)  }\\
&  \lesssim \Vert b_{2}\Vert_{LC_{q_{2},\lambda_{2}}^{\left \{  x_{0}\right \}
}}%
{\displaystyle \sum \limits_{j=1}^{\infty}}
\int \limits_{2^{j+1}r}^{2^{j+2}r}\left(  2^{j+1}r\right)  ^{-2n+\alpha
-1}\left(  1+\ln \frac{2^{j+1}r}{r}\right)  ^{2}\left \vert B_{2^{j+1}%
r}\right \vert ^{\lambda_{2}-\left(  \frac{1}{p_{1}}+\frac{1}{p_{2}}\right)
+2}\prod \limits_{i=1}^{2}\left \Vert f_{i}\right \Vert _{L_{p_{i}}\left(
B_{2^{j+1}r}\right)  }dt\\
&  \lesssim \Vert b_{2}\Vert_{LC_{q_{2},\lambda_{2}}^{\left \{  x_{0}\right \}
}}%
{\displaystyle \sum \limits_{j=1}^{\infty}}
\int \limits_{2^{j+1}r}^{2^{j+2}r}\left(  1+\ln \frac{2^{j+1}r}{r}\right)
^{2}\left \vert B_{2^{j+1}r}\right \vert ^{\lambda_{2}-\left(  \frac{1}{p_{1}%
}+\frac{1}{p_{2}}\right)  +2}\prod \limits_{i=1}^{2}\left \Vert f_{i}\right \Vert
_{L_{p_{i}}\left(  B_{2^{j+1}r}\right)  }\frac{dt}{t^{2n-\alpha+1}}\\
&  \lesssim \Vert b_{2}\Vert_{LC_{q_{2},\lambda_{2}}^{\left \{  x_{0}\right \}
}}\int \limits_{2r}^{\infty}\left(  1+\ln \frac{t}{r}\right)  ^{2}\left \vert
B_{t}\right \vert ^{\lambda_{2}-\left(  \frac{1}{p_{1}}+\frac{1}{p_{2}}\right)
+2}\prod \limits_{i=1}^{2}\left \Vert f_{i}\right \Vert _{L_{p_{i}}\left(
B_{t}\right)  }\frac{dt}{t^{2n-\alpha+1}}%
\end{align*}%
\begin{equation}
\lesssim \Vert b_{2}\Vert_{LC_{q_{2},\lambda_{2}}^{\left \{  x_{0}\right \}  }%
}\int \limits_{2r}^{\infty}\left(  1+\ln \frac{t}{r}\right)  ^{2}t^{n\lambda
_{2}-n\left(  \frac{1}{p_{1}}+\frac{1}{p_{2}}\right)  -1+\alpha}%
\prod \limits_{i=1}^{2}\left \Vert f_{i}\right \Vert _{L_{p_{i}}\left(
B_{t}\right)  }dt. \label{tk2}%
\end{equation}
Let $1<\tau<\infty$, such that $\frac{1}{q}=\frac{1}{q_{1}}+\frac{1}{\tau} $.
Then, by H\"{o}lder's inequality and (\ref{tk2}), we obtain%
\begin{align*}
G_{42}  &  =\left \Vert \left[  \left(  b_{1}-\left(  b_{1}\right)  _{B_{r}%
}\right)  \right]  T_{\alpha}^{\left(  2\right)  }\left[  f_{1}^{\infty
},\left(  b_{2}-\left(  b_{2}\right)  _{B_{r}}\right)  f_{2}^{\infty}\right]
\right \Vert _{L_{q}\left(  B_{r}\right)  }\\
&  \lesssim \left \Vert \left(  b_{1}-\left(  b_{1}\right)  _{B_{r}}\right)
\right \Vert _{L_{q_{1}}\left(  B_{r}\right)  }\left \Vert T_{\alpha}^{\left(
2\right)  }\left[  f_{1}^{\infty},\left(  b_{2}-\left(  b_{2}\right)
_{B}\right)  f_{2}^{\infty}\right]  \right \Vert _{L_{\tau \left(  B_{r}\right)
}}\\
&  \lesssim \Vert b_{1}\Vert_{LC_{q_{1},\lambda_{1}}^{\left \{  x_{0}\right \}
}}\Vert b_{2}\Vert_{LC_{q_{2},\lambda_{2}}^{\left \{  x_{0}\right \}  }}%
r^{\frac{n}{q}}\\
&  \times \int \limits_{2r}^{\infty}\left(  1+\ln \frac{t}{r}\right)
^{2}t^{n\left(  \lambda_{1}+\lambda_{2}\right)  -\frac{n}{q}+n\left(  \frac
{1}{q_{1}}+\frac{1}{q_{2}}\right)  -1}\prod \limits_{i=1}^{2}\left \Vert
f_{i}\right \Vert _{L_{p_{i}}\left(  B_{t}\right)  }dt.
\end{align*}

Similarly, $G_{43}$ has the same estimate above, here we omit the details,
thus the inequality%
\begin{align*}
G_{43}  &  =\left \Vert \left(  b_{2}-\left(  b_{2}\right)  _{B_{r}}\right)
T_{\alpha}^{\left(  2\right)  }\left[  \left(  b_{1}-\left(  b_{1}\right)
_{B_{r}}\right)  f_{1}^{\infty},f_{2}^{\infty}\right]  \right \Vert
_{L_{q}\left(  B_{r}\right)  }\\
&  \lesssim \Vert b_{1}\Vert_{LC_{q_{1},\lambda_{1}}^{\left \{  x_{0}\right \}
}}\Vert b_{2}\Vert_{LC_{q_{2},\lambda_{2}}^{\left \{  x_{0}\right \}  }}%
r^{\frac{n}{q}}\\
&  \times \int \limits_{2r}^{\infty}\left(  1+\ln \frac{t}{r}\right)
^{2}t^{n\left(  \lambda_{1}+\lambda_{2}\right)  -\frac{n}{q}+n\left(  \frac
{1}{q_{1}}+\frac{1}{q_{2}}\right)  -1}\prod \limits_{i=1}^{2}\left \Vert
f_{i}\right \Vert _{L_{p_{i}}\left(  B_{t}\right)  }dt
\end{align*}
is valid.

Finally, to estimate $G_{44}$, similar to the estimate of (\ref{tk2}), we have%
\begin{align*}
&  \left \vert T_{\alpha}^{\left(  2\right)  }\left[  \left(  b_{1}-\left(
b_{2}\right)  _{B_{r}}\right)  f_{1}^{\infty},\left(  b_{2}-\left(
b_{2}\right)  _{B_{r}}\right)  f_{2}^{\infty}\right]  \left(  x\right)
\right \vert \\
&  \lesssim%
{\displaystyle \sum \limits_{j=1}^{\infty}}
\left(  2^{j}r\right)  ^{-2n+\alpha}%
{\displaystyle \int \limits_{B_{2^{j+1}r}}}
\left \vert b_{1}\left(  y_{1}\right)  -\left(  b_{1}\right)  _{B_{r}%
}\right \vert \left \vert f_{1}\left(  y_{1}\right)  \right \vert dy_{1}%
{\displaystyle \int \limits_{B_{2^{j+1}r}}}
\left \vert b_{2}\left(  y_{2}\right)  -\left(  b_{2}\right)  _{B_{r}%
}\right \vert \left \vert f_{2}\left(  y_{2}\right)  \right \vert dy_{2}\\
&  \lesssim \Vert b_{1}\Vert_{LC_{q_{1},\lambda_{1}}^{\left \{  x_{0}\right \}
}}\Vert b_{2}\Vert_{LC_{q_{2},\lambda_{2}}^{\left \{  x_{0}\right \}  }}%
\int \limits_{2r}^{\infty}\left(  1+\ln \frac{t}{r}\right)  ^{2}t^{n\left(
\lambda_{1}+\lambda_{2}\right)  -\frac{n}{q}+n\left(  \frac{1}{q_{1}}+\frac
{1}{q_{2}}\right)  -1}\prod \limits_{i=1}^{2}\left \Vert f_{i}\right \Vert
_{L_{p_{i}}\left(  B_{t}\right)  }dt.
\end{align*}

Thus, we have%
\begin{align*}
G_{44}  &  =\left \Vert T_{\alpha}^{\left(  2\right)  }\left[  \left(
b_{1}-\left(  b_{1}\right)  _{B_{r}}\right)  f_{1}^{\infty},\left(
b_{2}-\left(  b_{2}\right)  _{B_{r}}\right)  f_{2}^{\infty}\right]
\right \Vert _{L_{p}\left(  B_{r}\right)  }\\
&  \lesssim \Vert b_{1}\Vert_{LC_{q_{1},\lambda_{1}}^{\left \{  x_{0}\right \}
}}\Vert b_{2}\Vert_{LC_{q_{2},\lambda_{2}}^{\left \{  x_{0}\right \}  }}%
r^{\frac{n}{q}}\\
&  \times \int \limits_{2r}^{\infty}\left(  1+\ln \frac{t}{r}\right)
^{2}t^{n\left(  \lambda_{1}+\lambda_{2}\right)  -\frac{n}{q}+n\left(  \frac
{1}{q_{1}}+\frac{1}{q_{2}}\right)  -1}\prod \limits_{i=1}^{2}\left \Vert
f_{i}\right \Vert _{L_{p_{i}}\left(  B_{t}\right)  }dt.
\end{align*}

By the estimates of $G_{4j}$ above, where $j=1$,$2$,$3$,$4$, we know that%
\begin{align*}
G_{4}  &  =\left \Vert T_{\alpha,\left(  b_{1},b_{2}\right)  }^{\left(
2\right)  }\left(  f_{1}^{\infty},f_{2}^{\infty}\right)  \right \Vert
_{L_{q}\left(  B_{r}\right)  }\lesssim \Vert b_{1}\Vert_{LC_{q_{1},\lambda_{1}%
}^{\left \{  x_{0}\right \}  }}\Vert b_{2}\Vert_{LC_{q_{2},\lambda_{2}%
}^{\left \{  x_{0}\right \}  }}r^{\frac{n}{q}}\\
&  \times \int \limits_{2r}^{\infty}\left(  1+\ln \frac{t}{r}\right)
^{2}t^{n\left(  \lambda_{1}+\lambda_{2}\right)  -\frac{n}{q}+n\left(  \frac
{1}{q_{1}}+\frac{1}{q_{2}}\right)  -1}\prod \limits_{i=1}^{2}\left \Vert
f_{i}\right \Vert _{L_{p_{i}}\left(  B_{t}\right)  }dt.
\end{align*}

Recalling (\ref{100}), and combining all the estimates for $G_{1},$ $G_{2}$,
$G_{3}$, $G_{4}$, we get%
\begin{align*}
\left \Vert T_{\alpha,\left(  b_{1},b_{2}\right)  }^{\left(  2\right)  }\left(
f_{1},f_{2}\right)  \right \Vert _{L_{q}\left(  B_{r}\right)  }  &
\lesssim \Vert b_{1}\Vert_{LC_{q_{1},\lambda_{1}}^{\left \{  x_{0}\right \}  }%
}\Vert b_{2}\Vert_{LC_{q_{2},\lambda_{2}}^{\left \{  x_{0}\right \}  }}%
r^{\frac{n}{q}}\\
&  \times \int \limits_{2r}^{\infty}\left(  1+\ln \frac{t}{r}\right)
^{2}t^{n\left(  \lambda_{1}+\lambda_{2}\right)  -\frac{n}{q}+n\left(  \frac
{1}{q_{1}}+\frac{1}{q_{2}}\right)  -1}\prod \limits_{i=1}^{2}\left \Vert
f_{i}\right \Vert _{L_{p_{i}}\left(  B_{t}\right)  }dt.
\end{align*}

Therefore, Theorem \ref{Teo 5} is completely proved.
\end{proof}

Now we can give the following theorem, which is another our main result.

\begin{theorem}
\label{teo15}(Our main result) Let $x_{0}\in{\mathbb{R}^{n}}$, $0<\alpha<mn,$
and $1\leq p_{i}<\frac{mn}{\alpha}$ for $i=1,\ldots,m$ such that $\frac{1}%
{q}=\sum \limits_{i=1}^{m}\frac{1}{p_{i}}+\sum \limits_{i=1}^{m}\frac{1}{q_{i}%
}-\frac{\alpha}{n}$ and $\overrightarrow{b}\in LC_{q_{i},\lambda_{i}%
}^{\left \{  x_{0}\right \}  }({\mathbb{R}^{n}})$ for $0\leq \lambda_{i}<\frac
{1}{n}$, $i=1,\ldots,m$. Let also, $T_{\alpha}^{\left(  m\right)  }$ $\left(
m\in%
\mathbb{N}
\right)  $ be a multilinear operator satisfying condition (\ref{4}), bounded
from $L_{p_{1}}\times \cdots \times L_{p_{m}}$ into $L_{q}$. If functions
$\varphi,\varphi_{i}:{\mathbb{R}^{n}\times}\left(  0,\infty \right)
\rightarrow \left(  0,\infty \right)  $ $\left(  i=1,\ldots,m\right)  $ and
$(\varphi_{1},\ldots,\varphi_{m},\varphi)$ satisfies the condition%
\begin{equation}%
{\displaystyle \int \limits_{r}^{\infty}}
\left(  1+\ln \frac{t}{r}\right)  ^{m}t^{n\left(  -\frac{1}{q}+%
{\displaystyle \sum \limits_{i=1}^{n}}
\lambda_{i}+%
{\displaystyle \sum \limits_{i=1}^{m}}
\frac{1}{q_{i}}\right)  -1}\operatorname*{essinf}\limits_{t<\tau<\infty}%
{\displaystyle \prod \limits_{i=1}^{m}}
\varphi_{i}(x_{0},\tau)\tau^{\frac{n}{p_{i}}}dt\leq C\varphi \left(
x_{0},r\right)  , \label{50}%
\end{equation}
where $C$ does not depend on $r$.

Then the operator $T_{\alpha,\overrightarrow{b}}^{\left(  m\right)  }$
$\left(  m\in%
\mathbb{N}
\right)  $ is bounded from product space $LM_{p_{1},\varphi_{1}}^{\{x_{0}%
\}}\times \cdots \times LM_{p_{m},\varphi_{m}}^{\{x_{0}\}}$ to $LM_{q,\varphi
}^{\{x_{0}\}}$. Moreover, we have
\begin{equation}
\left \Vert T_{\alpha,\overrightarrow{b}}^{\left(  m\right)  }\left(
\overrightarrow{f}\right)  \right \Vert _{LM_{q,\varphi}^{\{x_{0}\}}}\lesssim%
{\displaystyle \prod \limits_{i=1}^{m}}
\left \Vert \overrightarrow{b}\right \Vert _{LC_{q_{i},\lambda_{i}}^{\left \{
x_{0}\right \}  }}%
{\displaystyle \prod \limits_{i=1}^{m}}
\left \Vert f_{i}\right \Vert _{LM_{p_{i},\varphi_{i}}^{\{x_{0}\}}}. \label{51}%
\end{equation}

\end{theorem}

\begin{proof}
Similar to the proof of Theorem \ref{teo9}, For $1<p_{1},\ldots,p_{m}<\infty$,
since $(\varphi_{1},\ldots,\varphi_{m},\varphi)$ satisfies (\ref{50}) and by
(\ref{9}), we have%
\begin{align}
&  \int \limits_{r}^{\infty}\left(  1+\ln \frac{t}{r}\right)  ^{m}t^{n\left(
-\frac{1}{q}+%
{\displaystyle \sum \limits_{i=1}^{m}}
\lambda_{i}+%
{\displaystyle \sum \limits_{i=1}^{m}}
\frac{1}{q_{i}}\right)  -1}%
{\displaystyle \prod \limits_{i=1}^{m}}
\Vert f_{i}\Vert_{L_{p_{i}}(B_{t})}dt\nonumber \\
&  \leq \int \limits_{r}^{\infty}\left(  1+\ln \frac{t}{r}\right)  ^{m}\frac{%
{\displaystyle \prod \limits_{i=1}^{m}}
\left \Vert f_{i}\right \Vert _{L_{p_{i}}\left(  B_{t}\right)  }}%
{\operatorname*{essinf}\limits_{t<\tau<\infty}%
{\displaystyle \prod \limits_{i=1}^{m}}
\varphi_{i}(x_{0},\tau)\tau^{\frac{n}{p_{i}}}}\frac{\operatorname*{essinf}%
\limits_{t<\tau<\infty}%
{\displaystyle \prod \limits_{i=1}^{m}}
\varphi_{i}(x_{0},\tau)\tau^{\frac{n}{p_{i}}}}{t^{n\left(  \frac{1}{q}-%
{\displaystyle \sum \limits_{i=1}^{m}}
\lambda_{i}-%
{\displaystyle \sum \limits_{i=1}^{m}}
\frac{1}{q_{i}}\right)  }}\frac{dt}{t}\nonumber \\
&  \leq C%
{\displaystyle \prod \limits_{i=1}^{m}}
\left \Vert f_{i}\right \Vert _{LM_{p_{i},\varphi_{i}}^{\{x_{0}\}}}%
\int \limits_{r}^{\infty}\left(  1+\ln \frac{t}{r}\right)  ^{m}\frac
{\operatorname*{essinf}\limits_{t<\tau<\infty}%
{\displaystyle \prod \limits_{i=1}^{m}}
\varphi_{i}(x_{0},\tau)\tau^{\frac{n}{p}}}{t^{n\left(  \frac{1}{q}-%
{\displaystyle \sum \limits_{i=1}^{m}}
\lambda_{i}-%
{\displaystyle \sum \limits_{i=1}^{m}}
\frac{1}{q_{i}}\right)  +1}}dt\nonumber \\
&  \leq C%
{\displaystyle \prod \limits_{i=1}^{m}}
\left \Vert f_{i}\right \Vert _{LM_{p_{i},\varphi_{i}}^{\{x_{0}\}}}\varphi
(x_{0},r). \label{13}%
\end{align}
Then by (\ref{200}) and (\ref{13}), we get%
\begin{align*}
\left \Vert T_{\alpha,\overrightarrow{b}}^{\left(  m\right)  }\left(
\overrightarrow{f}\right)  \right \Vert _{LM_{q,\varphi}^{\{x_{0}\}}}  &
=\sup_{r>0}\varphi \left(  x_{0},r\right)  ^{-1}|B_{r}|^{-\frac{1}{q}%
}\left \Vert T_{\alpha,\overrightarrow{b}}^{\left(  m\right)  }\left(
\overrightarrow{f}\right)  \right \Vert _{L_{q}\left(  B_{r}\right)  }\\
&  \lesssim%
{\displaystyle \prod \limits_{i=1}^{m}}
\left \Vert \overrightarrow{b}\right \Vert _{LC_{q_{i},\lambda_{i}}^{\left \{
x_{0}\right \}  }}\sup_{r>0}\varphi \left(  x_{0},r\right)  ^{-1}\int
\limits_{r}^{\infty}\left(  1+\ln \frac{t}{r}\right)  ^{m}t^{n\left(  -\frac
{1}{q}+%
{\displaystyle \sum \limits_{i=1}^{m}}
\lambda_{i}+%
{\displaystyle \sum \limits_{i=1}^{m}}
\frac{1}{q_{i}}\right)  -1}%
{\displaystyle \prod \limits_{i=1}^{m}}
\Vert f_{i}\Vert_{L_{p_{i}}(B_{t})}dt\\
&  \lesssim%
{\displaystyle \prod \limits_{i=1}^{m}}
\left \Vert \overrightarrow{b}\right \Vert _{LC_{q_{i},\lambda_{i}}^{\left \{
x_{0}\right \}  }}%
{\displaystyle \prod \limits_{i=1}^{m}}
\left \Vert f_{i}\right \Vert _{LM_{p_{i},\varphi_{i}}^{\{x_{0}\}}}.
\end{align*}
Thus we obtain (\ref{51}). Hence the proof is completed.
\end{proof}

For the multi-sublinear commutator of the multi-sublinear fractional maximal
operator%
\[
M_{\alpha,\overrightarrow{b}}^{\left(  m\right)  }\left(  \overrightarrow
{f}\right)  \left(  x\right)  =\sup_{t>0}\left \vert B\left(  x,t\right)
\right \vert ^{\frac{\alpha}{n}}\int \limits_{B\left(  x,t\right)  }\frac
{1}{\left \vert B\left(  x,t\right)  \right \vert }%
{\displaystyle \prod \limits_{i=1}^{m}}
\left[  b_{i}\left(  x\right)  -b_{i}\left(  y_{i}\right)  \right]  \left \vert
f_{i}\left(  y_{i}\right)  \right \vert d\overrightarrow{y}%
\]
from Theorem \ref{teo15} we get the following new results.

\begin{corollary}
\label{corollary 3*}Let $x_{0}\in{\mathbb{R}^{n}}$, $0<\alpha<mn,$ and $1\leq
p_{i}<\frac{mn}{\alpha}$ for $i=1,\ldots,m$ such that $\frac{1}{q}%
=\sum \limits_{i=1}^{m}\frac{1}{p_{i}}+\sum \limits_{i=1}^{m}\frac{1}{q_{i}%
}-\frac{\alpha}{n}$ and $\overrightarrow{b}\in LC_{q_{i},\lambda_{i}%
}^{\left \{  x_{0}\right \}  }({\mathbb{R}^{n}})$ for $0\leq \lambda_{i}<\frac
{1}{n}$, $i=1,\ldots,m$ and $(\varphi_{1},\ldots,\varphi_{m},\varphi)$
satisfies condition (\ref{50}). Then, the operators $M_{\alpha,\overrightarrow
{b}}^{\left(  m\right)  }$ and $\overline{T}_{\alpha,\overrightarrow{b}%
}^{\left(  m\right)  }$ $\left(  m\in%
\mathbb{N}
\right)  $ are bounded from product space $LM_{p_{1},\varphi_{1}}^{\{x_{0}%
\}}\times \cdots \times LM_{p_{m},\varphi_{m}}^{\{x_{0}\}}$ to $LM_{q,\varphi
}^{\{x_{0}\}}$ for $p_{i}>1$ $\left(  i=1,\ldots,m\right)  $.
\end{corollary}

\begin{remark}
Note that, in the case of $m=1$ Theorem \ref{teo15} and Corollary
\ref{corollary 3*} have been proved in \cite{Gurbuz1, Gurbuz2}.
\end{remark}

\textbf{Acknowledgements}

The author would like to express his deep gratitude to the Assistant Professor
Turhan KARAMAN (Ahi Evran University, K\i r\c{s}ehir, TURKEY) for carefully
reading the manuscript and giving some valuable suggestions and important
comments during the process of this study.

\end{document}